\documentclass{amsart} 
\newif\ifpersonal
\usepackage{amsmath,amscd,amsthm,amssymb,mathrsfs,mathtools,bm} 
\usepackage{microtype,lmodern,textcomp} 
\usepackage{enumitem,comment,braket,xspace,tikz-cd} 
\usepackage[T1]{fontenc} 
\usepackage[centering,vscale=0.7,hscale=0.7]{geometry}
\usepackage[hidelinks]{hyperref}
\usepackage[capitalize]{cleveref}
\usepackage{graphicx}
\numberwithin{equation}{section}
\theoremstyle{plain}
\newtheorem{theorem}[equation]{Theorem}
\newtheorem{lemma}[equation]{Lemma}
\newtheorem*{definition*}{Definition} 
\newtheorem*{lemma*}{Lemma}

\newtheorem*{claim*}{Claim}

\newtheorem{proposition}[equation]{Proposition}
\newtheorem*{proposition*}{Proposition}
\newtheorem*{remark*}{Remark}
\newtheorem{corollary}[equation]{Corollary}
\theoremstyle{definition}
\newtheorem{definition}[equation]{Definition}
\newtheorem{definition-theorem}[equation]{Definition-Theorem}
\newtheorem{definition-lemma}[equation]{Definition-Lemma}

\newtheorem{notation}[equation]{Notation}

\newtheorem{remark}[equation]{Remark}
\numberwithin{equation}{section}

\ifpersonal
\newcommand{\personal}[1]{\textcolor[rgb]{0,0,1}{(Personal: #1)}}

\newcommand{\todo}[1]{\textcolor{red}{(Todo: #1)}}
\else
\newcommand{\personal}[1]{\ignorespaces}
\newcommand{\discussion}[1]{\ignorespaces}
\newcommand{\todo}[1]{\ignorespaces}
\fi

\providecommand{\abs}[1]{\lvert#1\rvert}

\newcommand{\bbA}{\mathbb A}

\newcommand{\bbC}{\mathbb C}

\newcommand{\bbN}{\mathbb N}

\newcommand{\bbQ}{\mathbb Q}
\newcommand{\bbR}{\mathbb R}

\newcommand{\bbZ}{\mathbb Z}

\newcommand{\cB}{\mathcal B}

\newcommand{\cE}{\mathcal E}
\newcommand{\cF}{\mathcal F}
\newcommand{\cH}{\mathcal H}
\newcommand{\cG}{\mathcal G}

\newcommand{\cL}{\mathcal L}
\newcommand{\cM}{\mathcal M}

\newcommand{\cO}{\mathcal O}

\newcommand{\cU}{\mathcal U}
\newcommand{\cV}{\mathcal V}

\makeatletter
\let\save@mathaccent\mathaccent
\newcommand*\if@single[3]{%
	\setbox0\hbox{${\mathaccent"0362{#1}}^H$}%
	\setbox2\hbox{${\mathaccent"0362{\kern0pt#1}}^H$}%
	\ifdim\ht0=\ht2 #3\else #2\fi
}
\newcommand*\rel@kern[1]{\kern#1\dimexpr\macc@kerna}
\newcommand*\widebar[1]{\@ifnextchar^{{\wide@bar{#1}{0}}}{\wide@bar{#1}{1}}}
\newcommand*\wide@bar[2]{\if@single{#1}{\wide@bar@{#1}{#2}{1}}{\wide@bar@{#1}{#2}{2}}}
\newcommand*\wide@bar@[3]{%
	\begingroup
	\def\mathaccent##1##2{%
		\let\mathaccent\save@mathaccent
		\if#32 \let\macc@nucleus\first@char \fi
		\setbox\z@\hbox{$\macc@style{\macc@nucleus}_{}$}%
		\setbox\tw@\hbox{$\macc@style{\macc@nucleus}{}_{}$}%
		\dimen@\wd\tw@
		\advance\dimen@-\wd\z@
		\divide\dimen@ 3
		\@tempdima\wd\tw@
		\advance\@tempdima-\scriptspace
		\divide\@tempdima 10
		\advance\dimen@-\@tempdima
		\ifdim\dimen@>\z@ \dimen@0pt\fi
		\rel@kern{0.6}\kern-\dimen@
		\if#31
		\overline{\rel@kern{-0.6}\kern\dimen@\macc@nucleus\rel@kern{0.4}\kern\dimen@}%
		\advance\dimen@0.4\dimexpr\macc@kerna
		\let\final@kern#2%
		\ifdim\dimen@<\z@ \let\final@kern1\fi
		\if\final@kern1 \kern-\dimen@\fi
		\else
		\overline{\rel@kern{-0.6}\kern\dimen@#1}%
		\fi
	}%
	\macc@depth\@ne
	\let\math@bgroup\@empty \let\math@egroup\macc@set@skewchar
	\mathsurround\z@ \frozen@everymath{\mathgroup\macc@group\relax}%
	\macc@set@skewchar\relax
	\let\mathaccentV\macc@nested@a
	\if#31
	\macc@nested@a\relax111{#1}%
	\else
	\def\gobble@till@marker##1\endmarker{}%
	\futurelet\first@char\gobble@till@marker#1\endmarker
	\ifcat\noexpand\first@char A\else
	\def\first@char{}%
	\fi
	\macc@nested@a\relax111{\first@char}%
	\fi
	\endgroup
}
\makeatother

\newcommand{\oD}{\widebar D}

\newcommand{\gr}{\operatorname{gr}}
\DeclareMathOperator{\DR}{DR}
\DeclareMathOperator{\RH}{RH}

\DeclareMathOperator{\FT}{FT}
\DeclareMathOperator{\Crit}{Crit}

\DeclareMathOperator{\im}{Im}
\DeclareMathOperator{\re}{Re}

\DeclareMathOperator{\cir}{\circ}

\newcommand{\ocH}{\overline{\cH}}

\newcommand{\an}{{\mathrm{an}}}
\newcommand{\formal}{{\mathrm{for}}}
\newcommand{\iso}{\mathrm{iso}}

\newcommand{\Conn}{\mathrm{Conn}}
\newcommand{\Stol}{\mathrm{Sto_<}}
\newcommand{\Sto}{\mathrm{Sto}}

\newcommand{\Constrz}{\mathrm{Constr_0}}
\newcommand{\Pervzero}{\mathrm{Perv}_0}
\newcommand{\Perv}{\mathrm{Perv}}

\newcommand{\SSh}{\mathrm{SSh}}

\newcommand{\hbarz}{\ocH_{z}}

\newcommand{\dbb}[1]{[\![#1]\!]}
\newcommand{\dbp}[1]{(\!(#1)\!)}

\DeclareMathOperator*{\colim}{colim}

\newcommand{\id}{\mathrm{id}}
\newcommand{\Id}{\mathrm{Id}}

\renewenvironment{abstract}{%
	\quotation
	\small
	\textbf{\textit{\abstractname.}} 
}{\endquotation}

\begin{document}

\title{Topological Laplace Transform and Decomposition of nc-Hodge Structures}
\author{Tony Yue Yu}
\address{Tony Yue YU, Department of Mathematics, M/C 253-37, Caltech, 1200 E.\ California Blvd., Pasadena, CA 91125, USA}
\email{yuyuetony@gmail.com}
\author{Shaowu Zhang}
\address{Shaowu Zhang, Department of Mathematics, M/C 253-23, Caltech, 1200 E.\ California Blvd., Pasadena, CA 91125, USA}
\email{szhang7@caltech.edu}
\date{May 29, 2024}

\maketitle
\begin{abstract}
We construct the topological Laplace transform functor from Stokes structures of exponential type to constructible sheaves on $\bbC$ with vanishing cohomology.
We show that it is compatible with the Fourier transform of $D$-modules, and induces an equivalence of categories.

We give two applications of the construction.

First, we study the Fourier transform of B-model nc-Hodge structures associated to Landau-Ginzburg models, and prove the compatibility between the $\bbQ$-structure and the Stokes structure from the connection.

Second, we relate the spectral decomposition of nc-Hodge structures to the vanishing cycle decomposition after Fourier transform via choices of Gabrielov paths.
This is motivated by the study of the atomic decomposition of A-model nc-Hodge structures associated to smooth projective varieties.
\end{abstract}

\tableofcontents
\section{Introduction}
The notion of nc-Hodge structure was introduced in \cite{KKP_Hodge_theoretic_aspect_of_mirror_symmetry} as a generalization of classical pure rational Hodge structure for non-commutative spaces.
It arises naturally in the study of mirror symmetry, enumerative geometry, as well as singularity theory.
Those nc-Hodge structures that come from geometry are expected to satisfy an additional condition called \emph{exponential type} (see \cite{Sabbah_kontsevich's_conjecture_on_an_algebraic_formula_for_vanishing_cycles_of_local_systems,Pomerleano_The_quantum_connection}).
An nc-Hodge structure of exponential type contains the so-called de Rham data and Betti data.
The de Rham part has the following two equivalent descriptions via Fourier transform:
\begin{enumerate}[wide]
    \item An algebraic vector bundle $H$ with a connection $\nabla$ on $\bbA^1\setminus 0$, having a singularity of exponential type at $0$ and a regular singularity at $\infty$.
    \item A regular holonomic $D$-module $M$ on $\bbA^1$ with regular singularities and vanishing de Rham cohomology.
\end{enumerate}

The corresponding descriptions of Betti data are:
\begin{enumerate}[wide, label=(\roman*)]
    \item A $\bbQ$-Stokes structure $(\cL, \cL_{\le})$ of exponential type.
    \item A constructible sheaf $\cF$ of $\bbQ$-vector spaces on $\bbC$ with vanishing cohomology.
\end{enumerate}

We construct a functor from (i) to (ii), and we show that it is compatible with Fourier transform and induces an equivalence of categories.
While functors from (ii) to (i) were previously considered in the literature (\cite{KKP_Hodge_theoretic_aspect_of_mirror_symmetry,Sabbah_introduction_to_Stokes_structure}), the forward direction was not known.

Our construction is motivated by the following two applications.

The first application concerns the B-model nc-Hodge structures associated to Landau-Ginzburg models.
The de Rham part is given as twisted de Rham cohomology with Gauss-Manin connection, and the Betti part is given by rapid decay cohomology.
We study its Fourier transform, and prove the compatibility between de Rham and Betti data, i.e.\ the filtration given by rapid decay cohomology fits the asymptotics of the solutions of the differential equations.
Furthermore, we compare this nc-Hodge structure with the alternative construction of the B-model nc-Hodge structure via gluing along Gabrielov paths in \cite[\S 3.2]{KKP_Hodge_theoretic_aspect_of_mirror_symmetry}. In particular, this implies that the gluing construction is independent of the Gabrielov paths chosen.

Second, we relate the spectral decomposition of nc-Hodge structures to the vanishing cycle decomposition after Fourier transform.

We show that for any non-anti-Stokes direction $\theta$, the asymptotic lift of the spectral decomposition along $\theta$ agrees with the vanishing cycle decomposition given by straight Gabrielov paths in the direction $\theta$.

Our study was motivated by the study of atomic decomposition of A-model nc-Hodge structures associated to smooth projective varieties.
Atomic decomposition is a non-archimedean version of spectral decomposition, which produces new birational invariants (see \cite{KKPY_Birational_invariants}).
It was originally conceived via the vanishing cycle decomposition assuming the convergence of nc-Hodge structures, and it was later reformulated in the non-archimedean setting bypassing the conjectural convergence property.

Below we give more details of our constructions and theorems.

\medskip
\paragraph*{\bfseries Topological Laplace transform}

Our goal is to give a geometric construction of the functor $\Phi$ in the following commutative diagram
\[\begin{tikzcd}[column sep=large]
  (H, \nabla) \rar[mapsto]{\Theta} \dar[mapsto] & M \dar[mapsto] \\
  (\cL, \cL_{\le}) \otimes_\bbQ \bbC \rar[mapsto]{\Phi\otimes\id_\bbC} & \cF\otimes_\bbQ \bbC \\
  (\cL, \cL_{\le}) \uar[mapsto] \rar[mapsto]{\Phi} & \cF \uar[mapsto]
\end{tikzcd}\]
where the upper vertical arrows are given by the Riemann-Hilbert correspondence, and the $\Theta$ is the composition of change of coordinate $z \mapsto z^{-1}$ and Fourier transform.

The inverse functor $\Psi \coloneqq \Phi^{-1}$ was studied in \cite[\S 2.3.2]{KKP_Hodge_theoretic_aspect_of_mirror_symmetry} and \cite[\S 7]{Sabbah_introduction_to_Stokes_structure}.
It is based on a classical result of Malgrange \cite[\S XI]{Malgrange91_equation_differentielle_a_coefficient_polynomiaux},
which states that for ordinary differential operators of exponential type, the Laplace transform of hypersolutions at the singularities of $\cF$ form the filtration of the Stokes structure $(\cL, \cL_{\le})$.

Our construction of the functor $\Phi$ consists of the following two steps:
\begin{enumerate}[wide]
  \item The category of Stokes structures $(\cL, \cL_\le)$ is equivalent to that of co-Stokes structures $(\cL, \cL_<)$.  
  Given a co-Stokes structure $(\cL, \cL_<)$ of exponential type, we build an $\bbR$-constructible subsheaf $\cG$ of the pullback of $\cL$ to $\bbC^*$.
  For any $\xi=\lambda e^{i\theta} \in \bbC$ with $\theta\in[0,2\pi)$ and $\lambda \in \mathbb{R}$, write\[S_\xi=\set{\lambda' e^{i\theta'}\in \bbC^*|\lambda'(\theta')=\exp(\re (\xi e^{-i\theta'}))}, \]
  see \cref{figure: S^1}. 
\begin{figure}
    \centering
    \includegraphics[width=0.35\textwidth]{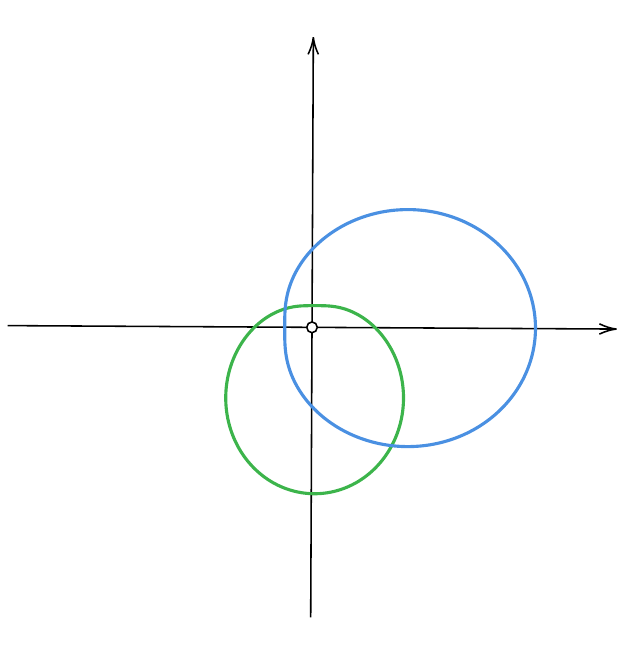}
    \caption{Examples of $S_{\xi}\subset\bbC^*$.}
    \label{figure: S^1}
\end{figure}
For any $z \in S_\xi$ with argument $\theta$, we have $\cG_z = \cL_{<\xi, \theta}$.
We show that $\cG$ retains all the information of the co-Stokes structure $(\cL, \cL_<)$ (see \cref{theorem: SFil=SSh}).
    
\item Next, we consider the diagram
\begin{equation*}
\begin{tikzcd}
      & \displaystyle\bigcup_{\xi \in \bbC} \{\xi\}\times S_{\xi} \subset \bbC \times \bbC^* \arrow[ld, "p_2"'] \arrow[rd, "q_2"] & \\
      \bbC & & \bbC^*
\end{tikzcd}
\end{equation*}
and show that the cohomology of $p_{2!}q_2^! \cG$ is concentrated in degree $0$ (see \cref{lemma: inverse Laplace only  has H^0}).
  We define $\cF = \cH^0(Rp_{2!}q_2^! \cG)$, and show that $\cF$ is a constructible sheaf with vanishing cohomology in \cref{Theorem: inverse Laplace transform is in constr}. In particular, we note that 
  \[\cF_\xi =H^1(S_\xi,\cG) = H^1(S^1,\cL_{<\xi}).\]
\end{enumerate}

\begin{theorem}[\cref{theorem: const0 and Stokes are equivalent 6.1}]\label{theorem: const0 and Stokes are equivalent 1.1}
    The functor $\Phi$ is an equivalence between the category $\Stol$ of co-Stokes structures of exponential type and the category $\Constrz$ of constructible sheaves on $\bbC$ with vanishing cohomology.
\end{theorem}

We deduce from the above equivalence an interesting corollary that any Stokes structure of exponential type on a trivial local system is trivial (see \cref{cor: trivial Stokes structure}).

\medskip

\paragraph*{\bfseries Equivalent descriptions of B-model nc-Hodge structure}

Next we apply \cref{theorem: const0 and Stokes are equivalent 1.1} to the study of B-model nc-Hodge structure associated to a Landau-Ginzburg model, $(Y, f)$, where $Y$ is a smooth complex quasi-projective variety, and $f$ is a proper function on $Y$.
It consists of a $\bbZ/2$-graded algebraic vector bundle over $\bbA^1$,
\[ H\coloneqq \mathbb{H}^\bullet\big(Y, (\Omega_Y^{\bullet}[u], ud-df )\big),\]
together with a $\bbQ$-structure given by the local system on $S^1$
\[ \cL_\theta \coloneqq  H^\bullet\big(Y,f^{-1}(\infty_\theta);\bbQ\big), \]
where $\infty_\theta$ denotes a point near infinity with argument $\theta$, see \cref{definition: B model nc-Hodge}.
We prove the following.

\begin{proposition} [\cref{lemma: comparison of Betti data}, \cref{proposition: de Rham and Betti data compatible}] 
The Betti data of the B-model nc-Hodge structure is given by the Stokes structure
\[\cL_{\leq r, \theta} = H^\bullet\big(Y,f^{-1}(\cH_{r,\theta});\bbQ\big),\]
where $\cH_{r,\theta}$ is the half-plane perpendicular to $\theta$ and shifted by $r$ (see \cref{notation: halfplanes}).
Under the topological Laplace transform, it is equivalent to the constructible sheaf $\cF = \cH^{-1}(\Pi ({}^p{\cH^\bullet}  Rf_*\bbQ))$, where $\Pi$ is the projector from $\Perv$ to $\Pervzero$ (see \cite[\S 12]{Katz_Exponential_sums_and_differential_equations}).
Furthermore, the Betti data is compatible with the de Rham data, i.e.\ it gives the Stokes filtration of the Gauss-Manin connection on twisted de Rham cohomology.
\end{proposition}

\begin{proposition} [\cref{gluing is dual to GM}, \cref{cor:B-model nc-Hodge gluging description}]
If we decompose the above B-model nc-Hodge structure by the gluing theorem in \cite[Theorem 2.35]{KKP_Hodge_theoretic_aspect_of_mirror_symmetry}, we obtain the description in Section 3.2 of loc.\ cit.
\end{proposition}

\medskip

\paragraph*{\bfseries Spectral decomposition versus vanishing cycle decomposition of nc-Hodge structures}

Decomposition of nc-Hodge structures can provide insight into semi-orthogonal decomposition of derived categories, which is difficult to study.
For example, atomic decomposition of A-model nc-Hodge structures gives rise to new birational invariants, previously conjectured from categorical considerations (see \cite{KKPY_Birational_invariants}).
This motivates the following comparison results.

An nc-Hodge structure of exponential type admits a formal decomposition according to the eigenvalues of the Euler vector field action, called spectral decomposition (see \cref{cor: spectral decomposition}).
We apply topological Laplace transform functor to the spectral decomposition of nc-Hodge structures, and we relate it to the vanishing cycle decomposition given by choices of Gabrielov paths.
As an intermediate step, we consider the decomposition given by the Stokes structure over a good open interval, which we explain more precisely in the following theorem.

\begin{proposition}[\cref{proposition: Stokes decomposition agrees with KKP}]
Let $(\cL, \cL_<)$ be a $\bbQ$-co-Stokes structure of exponential type, and $\cF$ its topological Laplace transform, which is a constructible sheaf of $\bbQ$-vector spaces on $\bbC$ with vanishing cohomology.
Let $c_1, \dots, c_n \in \bbC$ denote the singularities of $\cF$, or equivalently the exponents of $(\cL, \cL_<)$. Assume $\theta\in S^1$ is not an anti-Stokes direction of $(\cL, \cL_<)$, and $c_0 \in \bbC$ is a point with argument $\theta$ near infinity.
Then the following two decompositions of $\cF_{c_0}$ agree

\begin{enumerate}[wide]
  \item Stokes decomposition, induced by the trivialization of $(\cL, \cL_<)$ over $I_\theta$.
  \item Vanishing cycle decomposition, induced by straight Gabrielov paths $\gamma_i$ from $c_i$ to $c_0$,
  \begin{equation*}
    \cF_{c_0} =\bigoplus_{i = 1}^n  \cF_{c_0}/\psi_i(\cF_{c_i}),
  \end{equation*}
\end{enumerate}
where $\psi_i$ denotes the embedding of $\cF_{c_i}\longrightarrow \cF_{c_0}$ via $\gamma_i$.
\end{proposition}

\begin{theorem}[Theorem \ref{theorem: spectral decomposition and vanishing cycle decomposition}]
Let $(H, \nabla)$ be an nc-Hodge structure of exponential type.
The asymptotic lift of its spectral decomposition along a non-Stokes direction $\theta$ agrees with the vanishing cycle decomposition at a point near infinity with argument $\theta$.
\end{theorem}

\medskip

\paragraph*{\bfseries Related works}

We refer to \cite{KKP_Hodge_theoretic_aspect_of_mirror_symmetry, sabbah_non-commutative_hodge_structures} for the theory of nc-Hodge structures, and to \cite{Hertling_Frobenius_manifolds_and_moduli_spaces_for_singularities, Mochizuki_mixed_twistor_d_modules,F-bundle} for related concepts.

The topological Laplace transform on Stokes structures induced by Fourier transform on $D$-modules was studied in general by Malgrange \cite{Malgrange91_equation_differentielle_a_coefficient_polynomiaux} via a series of reduction steps, which does not lead to a direct geometric description.
The inverse functor $\Psi \coloneqq \Phi^{-1}$ is studied in Katzarkov-Kontsevich-Pantev \cite{KKP_Hodge_theoretic_aspect_of_mirror_symmetry}, Sabbah \cite[\S 7]{Sabbah_introduction_to_Stokes_structure}, and Mochizuki \cite{Mochizuki_Stokes_shells_and_fourier_transforms};
concrete calculations of Stokes matrices are given in D'Agnolo-Hien-Morando-Sabbah \cite{D'agnolo-Kashiwara-Riemann-Hilbert-holonomic} and Kapranov-Soibelman-Soukhanov \cite{Kapranov_Soibelman_Soukhanov_perverse_schobers_and_the_algebra_of_the_infrared}.

The B-model nc-Hodge structure was studied in \cite{Pham_la_descente_de_lefschetz_avec_} for the case of polynomial functions, and in \cite{sabbah_non-commutative_hodge_structures,sabbah_fourier_laplace_transform_of_a_variation_of_polarized_complex_hodge_structure} for the case of cohomologically tame functions on smooth affine algebraic varieties.

We refer to \cite{KKPY_Birational_invariants} for applications of the decomposition theorems in birational geometry.

\medskip
\paragraph*{\bfseries Acknowledgements}

Discussions with Ludmil Katzarkov, Maxim Kontsevich and Tony Pantev around the nc-Hodge theory provided invaluable motivations and perspectives for this study.
Special thanks to Claude Sabbah who generously answered our questions and pointed out helpful references in the literature to us.
We are grateful to Ruide Fu, Sheel Ganatra, Thorgal Hinault, Hiroshi Iritani, Peter Jossen, Shengxuan Liu, Matilde Marcolli, Yan Soibelman, Jean-Baptiste Teyssier, Zhaoting Wei, Yuchen Wu and Chi Zhang for inspiring discussions around the subject.
The authors were partially supported by NSF grants DMS-2302095 and DMS-2245099.

\section{Review of Stokes structures} \label{section: Stokes structures}

In this section, we recall the basic definitions of (co-)Stokes structures and perverse sheaves.

Stokes structures were introduced by Deligne and Malgrange in order to classify connections with an irregular singularity (see \cite{Malgrange1982_la_classification_des_connexions_irregulieres_a_une_variable,Malgrange_correspondence_singulartes_irreguliers}).
We refer to \cite{Sabbah_introduction_to_Stokes_structure} for a detailed account.
For convenience, we will mainly use a variant of Stokes structure called co-Stokes structure, which is equivalent.

\begin{notation}\label{notation: xi1 smaller than xi2}
For $\xi_1$ and $\xi_2\in \bbC$, we denote  
\begin{align*}
    \xi_1\leq_\theta  \xi_2 &\text{ if } \re(\xi_1e^{-i\theta})< \re(\xi_2e^{-i\theta}) \text{ or } \xi_1=\xi_2.\\ \xi_1<_\theta  \xi_2 &\text{ if } \re(\xi_1e^{-i\theta})< \re(\xi_2e^{-i\theta}).
\end{align*}
\end{notation}

Let $k$ be any field.

\begin{definition} 
  \label{def: Stokes structure}
  A \emph{$k$-Stokes structure of exponential type} $(\cL, \cL_{\leq })$ with exponents $c_1, \dots,c_n \in \bbC$ consists of a local system $\cL$ on $S^1$ of finite dimensional $k$-vector spaces and a family of subsheaves $\cL_{\leq \xi}$ indexed by $\xi\in \bbC$ satisfying the following conditions:
  \begin{enumerate}[wide]
      \item For any $\theta\in S^1$, $\xi_1\leq_\theta  \xi_2$ implies that $ \cL_{\leq \xi_1, \theta}\subset \cL_{\leq \xi_2, \theta}$.
      \item There exist local systems $\gr_{c_1}\cL, \dots, \gr_{c_n}\cL $ on $S^1$, such that if we denote
      \begin{align*}
      \gr\cL&\coloneqq\bigoplus_{i=1}^n\gr_{c_i}\cL, \\         (\gr \cL)_{\leq \xi}&\coloneqq\bigoplus_{i=1}^n j_{\xi!}j_\xi^*\gr_{c_i}\cL, \text{ where } j_\xi \colon \set{\theta\in S^1 | c_i\leq_\theta  \xi} \hookrightarrow S^1,
      \end{align*}      
      then locally on $S^1$, there exists an isomorphism of filtered local systems between $(\cL, \cL_{\leq})$ and $\big(\gr\cL,(\gr \cL)_{\leq}\big)$.
  \end{enumerate}
\end{definition}

\begin{definition} \label{definition: co-Stokes structure}
  A \emph{$k$-co-Stokes structure of exponential type} $(\cL, \cL_<)$ with exponents $c_1, \dots,c_n \in \bbC$  is defined as in \cref{def: Stokes structure} where every symbol $\leq$ is replaced by $<$.
\end{definition}

Let $\Sto_\leq$ (resp.\ $\Stol$) denote the category of Stokes (resp.\ co-Stokes) structures of exponential type.
The following lemma follows from Definition 1.34 and Remark 1.41 of \cite{Sabbah_introduction_to_Stokes_structure}.

\begin{lemma}\label{coStokes=Stokes}
    The category of co-Stokes structures of exponential type is equivalent to the category of Stokes structures of exponential type.
\end{lemma}
\begin{definition} \label{definition: Stokes direction and small}
Let $(\cL, \cL_<)$ be a co-Stokes structure of exponential type with exponents $\{c_1, \dots,c_n\}$. 
We call $\theta\in S^1$ a \emph{Stokes direction} if for some $c_i\neq c_j$, we have \[\re(c_i{e^{-i\theta}})=\re(c_je^{-i\theta}).\]
We call $\theta\in S^1$ an \emph{anti-Stokes direction} if for some $c_i\neq c_j$, we have \[\im(c_i{e^{-i\theta}})=\im(c_je^{-i\theta}).\]An open interval $I\subset S^1$ is called \emph{good} (resp.\ \emph{small}) if it contains exactly one (resp.\ at most one) Stokes direction for each pair $c_i\neq c_j$.
\end{definition}

The next proposition shows that co-Stokes structures trivialize over a small open interval and such a trivialization is unique over a good open interval.
\begin{proposition}\label{proposition: Stokes filtration trivializes over good open intervals}
Let $(\cL, \cL_<)$ be a Stokes structure of exponential type with exponents $\{c_1, \dots,c_n\}$.
\begin{enumerate}[wide]
    \item On any small open interval $I\subset S^1$, there exists an isomorphism $\cL|_I \xrightarrow{\ \sim \ } \bigoplus_{c_i} \gr_{c_i} \cL|_I$ compatible with filtration.
    \item \label{over good is unique} On any good open interval $I\subset S^1$, there exists a unique isomorphism $\cL|_I \xrightarrow{\ \sim \ } \bigoplus_{c_i} \gr_{c_i} \cL|_I$ compatible with filtration.
    \item On any good open interval $I\subset S^1$, the isomorphism of sheaves in $(\ref{over good is unique})$ induces a decomposition of local sections\begin{equation*}
        \cL(I)\xrightarrow{\ \sim \ } \bigoplus_{c_i} \gr_{c_i} \cL(I).
    \end{equation*} 
    The preimage of $\gr_{c_i} \cL(I)$ is $\cL_{\leq c_i}(I)$.  
    \item Let $\lambda\colon (\cL, \cL_{<}) \xrightarrow{\ \ \  }(\cL^{\prime}, \cL_{<}^{\prime})$ be a morphism of co-Stokes structures. Then, for any good open interval $I \subset S^1$, the morphism $\lambda|_I$ is graded with respect to the splittings in (\ref{over good is unique}).
    
\end{enumerate}
\end{proposition}
\begin{proof} 
See \cite{Malgrange1982_la_classification_des_connexions_irregulieres_a_une_variable} and \cite[Corollary 3.15]{Sabbah_introduction_to_Stokes_structure} for (1) and (4). See \cite[Proposition 2.2 and Remark 2.3]{Hertling_Sabbah_examples_of_non-commutative_hodge_structures} for (2) and (3).
\end{proof}

The trivialization of co-Stokes structure over a small open interval allows us to compute sheaf cohomology. The following corollary will be helpful for our later construction.

\begin{corollary}\label{corollary: H1(small)=0}
    Let $(\cL, \cL_{<})$ be a co-Stokes structure of the exponential type. Let $I$ be a small open interval. Then we have \begin{equation*}
    H^1(I, \cL_{<\xi})=0.
    \end{equation*}
\end{corollary}
\begin{proof}
See, for example, \cite[Lemma 3.12]{Sabbah_introduction_to_Stokes_structure}. 
\end{proof}

\section{Co-Stokes sheaves}\label{section: assemble Stokes}

Since we will express topological Laplace transforms purely in terms of sheaf operations, we need to encode the data of a co-Stokes structure of exponential type into an $\bbR$-constructible sheaf on $\bbC^*$, which we call a \emph{co-Stokes sheaf} (see \cref{definition: co-Stokes structure}).

We construct a functor from co-Stokes structures of exponential type to co-Stokes sheaves in \cref{lemma: sigma functor}, and an inverse functor in \cref{lemma: tau functor}.
We show that they give an equivalence of categories (see \cref{theorem: SFil=SSh}).

\begin{notation}\label{notation: circles}        
For $\xi=\lambda e^{i\theta} \in \bbC$ with $\theta\in[0,2\pi)$ and $\lambda \in \mathbb{R}$, denote
\begin{align*}
    S_\xi&=\set{\lambda' e^{i\theta'}\in \bbC^*|\lambda'(\theta')=\exp(\re (\xi e^{-i\theta'}))},\\ 
    D_\xi&=\set{\lambda' e^{i\theta'}\in \bbC^*|\lambda'(\theta')<\exp(\re (\xi e^{-i\theta'}))},
\end{align*}
(see Fig.\ \ref{figure: S^1}) and denote a parametrization by
\begin{align*}
i_\xi\colon  S^1 &\longrightarrow S_\xi\subset \bbC^* \\
\theta &\longmapsto e^{\re(\xi e^{-i\theta})}e^{i\theta} .
\end{align*}
For $z = \lambda' e^{i\theta'}\in \bbC^*$, we denote ${S_\xi {<} z}$ if $z$ lies outside the closure of $D_\xi$, i.e.\ $ \lambda'>\exp^{\re(\xi e^{-i\theta'})}$.
\end{notation}

\begin{definition}\label{definition: co-Stokes sheaf}
    A \emph{co-Stokes sheaf $(\cL, \cG)$ with exponents $\{c_1, \dots,c_n\}$} consists of a local system $\cL$ on $\bbC^*$, and a subsheaf $\cG$ of $\cL$ such that each point $z\in\bbC^*$ has an open neighborhood $U$ satisfying the following two conditions:
    \begin{enumerate}[wide]
        \item The restriction $\cL|_{U}$ is a constant sheaf.
        \item There exists a decomposition \begin{equation*}
    \cL|_{U} = \bigoplus_{i=1}^n\gr_{c_i}\cL,
    \end{equation*} 
    where each $\gr_{c_i}\cL$ is a constant sheaf on $U$ and for all $z' \in U$
    \begin{equation*}
    \cG_{z'} = \bigoplus_{S_{c_i}{<}z'}\gr_{c_i}\cL .
    \end{equation*}
    \end{enumerate}
    A morphism of co-Stokes sheaves from $(\cL, \cG)$ to $(\cL', \cG')$ is a morphism of local systems from $\cL$ to $\cL'$ that restricts to a morphism of subsheaves from $\cG$ and $\cG'$.
    Let $\SSh$ denote the category of co-Stokes sheaves.
\end{definition}

\begin{notation}\label{notation: circle-infty}
Given a co-Stokes structure or a co-Stokes sheaf with exponents $\{c_1, \dots, c_n\}$, let $S_\infty\subset\bbC^*$ denote a circle centered at $0$ encircling all $S_{c_i}$ for $1\leq i\leq n$, and $i_\infty\colon S^1\longrightarrow S_\infty$ the homeomorphism given by rescaling.
\end{notation}

Let $\pi \colon \bbC^* \xrightarrow{\ \sim \ } S^1 \times \bbR^1 \xrightarrow{\ } S^1$ denote the composition, where the first arrow is given by $r e^{i \theta} \mapsto (\theta, \log r)$, and the second arrow is the projection.

\begin{lemma}\label{lemma: assemble Stokes filtratoin}
Given a co-Stokes structure of exponential type $(\cL, \cL_{<})$ on $S^1$, there exists a unique subsheaf $\cG$ of $\pi^*\cL$ on $\bbC^*$ such that for any $z\in \bbC^*$ with argument $\theta$, lying on a circle $S_\xi$ for some $\xi\in\bbC$, we have
\begin{equation*}
  \cG_z = \cL_{<\xi, \theta}.
\end{equation*}
\end{lemma}

\begin{proof}
We construct the subsheaf $\cG$ by specifying its stalks $\cG_z$ for all $z\in \bbC^*$. First observe that if $z$ lies on the intersection of two circles $S_{\xi_1}$ and $S_{\xi_2}$, then we have the relation $\re(\xi_1 e^{-\theta})=\re (\xi_2 e^{-i\theta})$. So the two subspaces $\cL_{<\xi_1, \theta}$ and $\cL_{<\xi_2, \theta}$ are the same. Hence, the definition of $\cG_z$ is independent of the choice of $\xi$.

It remains to show that the $\cG_z$ forms a subsheaf. For this, we consider the étale space $\mathscr{L}$ on $\bbC^*$ associated to the local system $\pi^*\cL$ on $\bbC^*$. We will prove that the collection of subspaces $\mathscr{G}_z$ of $\mathscr{L}_z$, defined by $\mathscr{G}_z=\cL_{<\xi, \theta}$, is an open set in the étale topology. 

By the local gradedness of the co-Stokes filtration, for each $\phi_z\in \mathscr{G}_z$ we have $\phi_z = \sum l_{c_j, \theta}$ for some elements $l_{c_j, \theta}\in \gr_{c_j}\cL_{\theta}$ and $c_j <_\theta  z$. Then there exists an open neighborhood $U$ of $z$, such that for all $ z'=r'e^{i\theta'}\in U$, we have  $ c_j<_{\theta'}z'$. Thus, the sum  $\phi_z = \sum l_{c_j, \theta}$ extends to a sum of local section  $\phi = \sum l_{c_j}$ over $U$, and $\set{\phi_{z'}| z'\in U}$ defines an open neighborhood of $\phi_z$ in $\mathscr{G}$. We conclude $\mathscr{G}$ is an open set of $\mathscr{L}$ and it corresponds to a subsheaf $\cG$.
\end{proof}

\Cref{lemma: assemble Stokes filtratoin} allows us to assemble a co-Stokes structure on $S^1$ into a co-Stokes sheaf on $\bbC^*$.

\begin{lemma}\label{lemma: sigma functor}
We have a functor from the category of co-Stokes structures of exponential type on $S^1$ to the category of co-Stokes sheaves  on $\bbC^*$,
\begin{align*}
    \sigma\colon  \Stol &\longrightarrow \SSh\\ (\cL, \cL_{<})&\longmapsto ( \pi^*\cL, \cG)
\end{align*}
where $\cG$ is given in \cref{lemma: assemble Stokes filtratoin}.
\end{lemma}
\begin{proof}
A morphism $\psi$ between two co-Stokes structures $(\cL, \cL_{<})$ and $(\cL', \cL'_{<})$ induces a morphism between local systems $\pi^*\cL$ and $\pi^*\cL'$ on $\bbC^*$. Since $\psi$ preserves the filtrations, it restricts to a morphism between subsheaves $\cG$ and $\cG'$.
So the functor is well-defined.
\end{proof}

Next we show how a co-Stokes sheaf on $\bbC^*$ can be disassembled into a co-Stokes structure on $S^1$.
\begin{notation}\label{notation: Iz Rz}
For $z\in \bbC^*$, denote $R_z \coloneqq \{\bbR_{\ge 1}\cdot z \}\subset \bbC^*$ and $I_z  \coloneqq \{\bbR_{> 0}\cdot z \}\subset \bbC^*$. Let $\pi_z\colon  S_{z}\times [0, \infty) \to S_{z}\times \{0\}$ denote the projection, and $j\colon  S_{z}\times [0, \infty) \hookrightarrow S_{z}\times \bbR = \bbC^*$ the inclusion.
\end{notation}

\begin{lemma}\label{lemma: tau functor}
We have a functor from the category of co-Stokes sheaves on $\bbC^*$ to the category of co-Stokes structures of exponential type on $S^1$,
\begin{align*}
\tau\colon  \SSh&\longrightarrow \Stol \\
(\cL, \cG)&\longmapsto (\cL', \cL'_{<\xi})\coloneqq(i^*_\infty\cL,i_{\xi}^*\cG),
\end{align*}
where $i_\infty$ and $i_\xi$ are defined in Notations \ref{notation: circle-infty} and \ref{notation: circles}.
\end{lemma}

\begin{proof}
First, the inclusion $i^*_\xi\cG \hookrightarrow i^*_\infty\cL$ is given by the following commutative diagram. 
\begin{equation*}
\begin{tikzcd}
i_\xi^*\cG \arrow[d] & i_\xi^*\pi_z^*R^0\pi_{z*}j^*\cG \arrow[d] \arrow[l, "\sim"'] \arrow[r, equal] & i_\infty^*\pi_z^*R^0\pi_{z*}j^*\cG \arrow[d] \arrow[r] & i_\infty^*j^*\cG \arrow[d] \\
i_\xi^*\cL           & i_\xi^*\pi^*_zR^0\pi_{z*}j^*\cL \arrow[l, "\sim"] \arrow[r, equal]            & i_\infty^*\pi_z^*R^0\pi_{z*}j^*\cL \arrow[r]           & i_\infty^*j^*\cL          
\end{tikzcd}
\end{equation*}
Let $z=i_\xi(\theta)$ and $x \in S_\infty\cap R_z$. Then the inclusion $(i_\xi^*\cG)_\theta\hookrightarrow(i_\infty^*\cL)_\theta$ is the composition of $\cG_z\xleftarrow{\ \sim \ } \cG(R_z)\xrightarrow{\ \sim \ } \cG_x=\cL_x$.

Next, we prove that $(\cL', \cL'_{<})$ is a co-Stokes structure. Since $(\cL, \cG)$ is a co-Stokes sheaf, at each $z$, there exists an open neighborhood $U\subset \bbC^*$ and constant sheaves $\gr_{c_i}\cL$ such that for any $z'=r'e^{i\theta'}\in U$, we have 
\begin{equation*}
    \cL|_{U} = {\bigoplus}_{i=1}^n\gr_{c_i}\cL \quad\text{and}\quad\cG_{z'} = \bigoplus_{S_{c_i}<_{\theta'}z'}\gr_{c_i}\cL.
\end{equation*}
By locally extending these sheaves $\gr_{c_i}\cL$, we can assume that $U$ is a good open sector containing $z$ (see \cref{definition: small and good sectors}). Let $I\subset S^1$ be the open interval corresponding to $U$. Over $I\subset S^1$, we have a decomposition of $\cL'$ as: \begin{equation*}
    \cL'_I=(i_\infty^*\cL)_I = i_\infty^*(\cL|_{U}) = i_\infty^*\bigoplus_{i=1}^n \gr_{c_i}\cL.
\end{equation*}
For any $\theta'\in S^1$ near $\theta$, we have a decomposition\begin{equation*}
   (\cL_{<\xi}')_{\theta'}= (i^*_\xi\cG)_{\theta'} = \cG_{i_\xi(\theta')} = \bigoplus_{S_{c_i}{<}i_\xi(\theta')}\gr_{c_i}\cL=\bigoplus_{c_i<_{\theta'}\xi
}\gr_{c_i}\cL.
\end{equation*}
So $(\cL', \cL'_{<})$ is indeed a co-Stokes structure.

Now, we prove that any morphism  $f$ of co-Stokes sheaves from $(\cL_1, \cG_1)$ to $(\cL_2, \cG_2)$ defines a morphism of co-Stokes structures \[\tau(f)\colon  (i^*_\infty\cG_1,i_\xi^*\cG_1)\longrightarrow (i^*_\infty\cG_2,i_\xi^*\cG_2).\] 
A map of local systems $\cL_1\rightarrow \cL_2$ on $\bbC^*$ induces a map on local systems on $S^1$ by restriction 
\begin{equation*}
    i_\infty^*(f)\colon  i^*_\infty\cL_1=i^*_\infty\cG_1\longrightarrow i_\infty^*\cL_2=i_\infty^*\cG_2.
\end{equation*} 
For each filtration indexed by $z$,  we consider the following diagram
\begin{equation*}
    \begin{tikzcd}
i_\infty^*j^*\cG_1 =\cL_1\arrow[r]                       & i_\infty^*j^*\cG_2 =\cL_2                      \\
i_\infty^*\pi_z^* R^0\pi_{z*} j^*\cG_1 =\cL_{1,<z}\arrow[r] \arrow[u] & i_\infty^*\pi_z^* R^0\pi_{z*} j^*\cG_2=\cL_{2,<z} \arrow[u].
\end{tikzcd}
\end{equation*}
By the naturality of $\pi_z^*R^0\pi_{z*}\rightarrow \id$, this diagram commutes. Hence the map of the local system is compatible with filtration. This completes the proof.
\end{proof}

\begin{lemma}\label{Lemma: tau circ sigma =id}
There is a natural isomorphism $\tau \cir \sigma \xrightarrow{\ \sim \ } \id$.
\end{lemma} 
\begin{proof}
We prove $(\cL', \cL_{<}')=(\tau\cir\sigma)(\cL, \cL_<)$ is isomorphic to $(\cL, \cL_<)$. 

Since $i_\infty$ is a section of $\pi$, we have by construction \[
 \cL'=i_\infty^*\pi^*\cL= (\pi\cir i_\infty)^*\cL=\cL. \]
 For each subsheaf $\cL_{<\xi}'$, we have by definition
 \[\cL_{<\xi', \theta}=i^*_\xi(\sigma(\cL, \cL_<))_\theta = \cL_{<\xi, \theta}.
\] This implies that the subsheaves $\cL'_{<\xi}$ and $\cL_{<\xi}$ are isomorphic as well.
\end{proof}
\begin{lemma}\label{Lemma: sigma circ tau =id}
There is a natural isomorphism $\sigma \cir \tau \xrightarrow{\ \sim \ } \id$.
\end{lemma}
\begin{proof}
Let $(\cL', \cG')\coloneqq (\sigma\cir \tau)(\cL, \cG)$, We construct an isomorphism between the two local systems $\cL$ and $\cL'$ by adjunction:
\[\cL'=\pi^*i_\infty^*\cL=\pi^*R^0\pi_* i_{\infty*}i^*_\infty\cL \xleftarrow{\ \sim \ } \pi^*R^0\pi_* \cL\xrightarrow{\ \sim \ } \cL.
\]
For each point $z\in \bbC^*$, denote $x \in S_\infty \cap R_z$.
Then the above isomorphism of local systems at $z$ is
\[
    \psi_z\colon   \cL'_z=\cL_x \xleftarrow{\ \sim \ } \cL(I_z)\xrightarrow{\ \sim \ } \cL_z.
\]
By \cref{lemma: tau functor}, the inclusion $\cG'\hookrightarrow \cL'$  at $z$ is given by the following restriction maps
\begin{equation*}
        \phi_z\colon  \cG_z'=\cG_z\longrightarrow \cL_z \xleftarrow{\ \sim \ } \cL(R_z)\xrightarrow{\ \sim \ } \cL_x=\cL'_z.
\end{equation*}
Then the image of $\cG'_z$ under $\psi_z$ is $(\psi_z\cir \phi_z)(\cG'_z)\subset \cL_z$, By considering the following commutative diagram,
\begin{equation*}
   \begin{tikzcd}
                         & \cL_z \arrow[d, equal]      & \cL(I_z) \arrow[l,"\sim"'] \arrow[r,"\sim"] \arrow[d,"\sim"'] & \cL_x \arrow[d, equal] \\ 
\cG'_z=\cG_z\arrow[r,hook] &\cL_z                     & \cL(R_z) \arrow[r,"\sim"'] \arrow[l,"\sim"]                    & \cL_x                               
\end{tikzcd}
\end{equation*}
we see that $(\psi_z\cir \phi_z)(\cG'_z)=\cG_z\subset \cL_z$.
We conclude that the isomorphism of local systems preserves the subsheaves, completing the proof.
\end{proof}

Combining Lemmas \ref{Lemma: tau circ sigma =id} and \ref{Lemma: sigma circ tau =id}, we obtain the following theorem:

\begin{theorem}\label{theorem: SFil=SSh}
    The functor $\sigma\colon \Stol\longrightarrow \SSh$ is an equivalence of category with $\tau$ as its quasi-inverse.
\end{theorem}

\section{From constructible sheaves to co-Stokes structures}\label{section: From constructible sheaves to co-Stokes structures}

In \cref{subsection: laplace transform to co-Stokes structure}, we describe topological Laplace transform functor $\Psi$ from constructible sheaves $\cF$ on $\bbC$ with vanishing cohomology to co-Stokes structures $(\cL, \cL_<)$ of exponential type.
Composing with the equivalence in \cref{theorem: SFil=SSh}, we obtain a functor $\sigma \cir \Psi$ from constructible sheaves to co-Stokes sheaves.
In \cref{subsection: assemble Laplace transform}, we give a direct sheaf-theoretic construction of the composite functor, bypassing co-Stokes structures.

The functor $\Psi$ was studied in \cite[\S 2.3.2]{KKP_Hodge_theoretic_aspect_of_mirror_symmetry} and \cite[\S 7]{Sabbah_introduction_to_Stokes_structure}.
It is shown in \cite[Theorem 2.29]{KKP_Hodge_theoretic_aspect_of_mirror_symmetry} that the category of constructible sheaves on $\bbC$ with vanishing cohomology is equivalent to the category of perverse sheaves on $\bbC$ with vanishing cohomology, via the functors
\[\cH^{-1}\colon  \Pervzero(k) \longrightarrow  \Constrz(k)\quad \text{and} \quad [1]\colon  \Constrz(k) \longrightarrow \Pervzero(k) .\]
For the compatibility between $\Psi$ and $D$-module Fourier transform, we refer to \cite[\S XI]{Malgrange91_equation_differentielle_a_coefficient_polynomiaux} for the case of $\Pervzero$ and \cite[\S VII]{Sabbah_introduction_to_Stokes_structure} for the general case.

\subsection{Topological Laplace transform to co-Stokes structures} \label{subsection: laplace transform to co-Stokes structure}

We will often write $z=e^{\lambda}e^{i\theta}$ for points in $\bbC^*$, where co-Stokes sheaves are considered,
and write $\xi=\lambda e^{i\theta}$ for points in $\bbC$, where constructible sheaves are considered.    

\begin{notation}\label{notation: halfplanes}
For $\theta\in [0,2\pi)$, $\xi\in \bbC$ and $\lambda\in \mathbb{R}$, we denote 
\begin{align*}
    \cH_{\theta, \lambda}&\coloneqq
\set{x\in \bbC | \re ( xe^{-i\theta})>\lambda}\subset \bbC\\ l_z &\coloneqq \set{x\in \bbC | \re ( xe^{-i\theta})=\lambda}\subset \bbC.
\end{align*}
For $z=e^\lambda e^{i\theta}\in\bbC^*$, with $\theta \in [0,2\pi)$ and $\lambda\in\mathbb{R}$, we denote $\cH_z\coloneqq\cH_{\theta, \lambda}$ and call the point $\xi = re^{i\theta}$ the \emph{center} of $\ocH_z$.
\end{notation}
We remark that $\xi \in l_z$ if and only if $z\in S_\xi$ (see \cref{notation: circles} and \cref{remark: alternative description of unions of S1}). 

\begin{proposition}\label{Definition: Laplace transform}
For any $\xi \in \bbC$ and $\lambda >\re(\xi e^{-i\theta})$, consider the following diagram. 
\[\begin{tikzcd}
           & {\bigcup_{\theta\in S^1} \cH_{\theta, \lambda}\times \{\theta\}} \arrow[d, "j"] \arrow[ldd, "p_\lambda"', bend right] \arrow[rdd, "q_\lambda", bend left] &     \\
           & {\bigcup_{\theta\in S^1} \cH_{\theta, \re(\xi e^{-i\theta})}\times \{\theta\}} \arrow[rd, "q_\xi"] \arrow[ld, "p_\xi"']                                          &     \\
\bbC &                                                                                                                                                                 & S^1
\end{tikzcd}\]
We have a functor
\begin{align*}
    \Psi \colon  \Constrz &\longrightarrow  \Stol \\
    \cF &\longmapsto (\cL, \cL_<)
\end{align*}
with $\cL = \colim_{\lambda\rightarrow +\infty}  R^0q_{\lambda*}p_\lambda^*\cF$ and $\cL_{<\xi} = R^0q_{\xi*}p_\xi^*\cF$.
\end{proposition}
\begin{proof}
Given a constructible sheaf $\cF$ with singularities $\{c_1, \dots,c_n\}$, when $\lambda$ is large enough such that for any $\theta \in [0,2\pi)$, the closed half-plane $\ocH_{\theta, \lambda}$ does not contain any $c_i$, we have $R^0q_{\lambda *}p_\lambda^*\cF\xrightarrow{\ \sim \ }\cL$. 
Now we fix such a $\lambda$.
Let us show that $\Psi(\cF)$ is a co-Stokes structure.
For any $\theta\in S^1$, we need to show that, locally around $\theta$ there exists a decomposition of $\cL$ which is compatible with the filtrations $\cL_<$. 
Let $c_0$ be a general point in $\cH_{\theta, \lambda}$. We have \[
\cL_\theta=\cF\big(\ocH_{\theta, \lambda}\big)
\xrightarrow{\ \sim \ }\cF_{c_0},
\]
For $\xi\in\bbC$ and $\theta'$ near $\theta$, we have the restriction map
\[
\cL_{<\xi, \theta'}=\cF\big(\ocH_{\theta', \re(\xi e^{-i\theta'})}\big)\hookrightarrow\cF_{c_0},
\]
which is an inclusion by \cite[p.\ 41]{KKP_Hodge_theoretic_aspect_of_mirror_symmetry}.
Consider a vanishing cycle decomposition of $\cF$ at $c_0$ (see \cref{def: vanishing cycle decomposition})
\[\cF_{c_0}\xrightarrow{\ \sim \ }\bigoplus_{i=1}^n\cF_{c_0}/\cF_{c_i}.\]
This induces a decomposition
\[
 \cL_{<\xi, \theta'} \xrightarrow{\ \sim \ }  \bigoplus_{c_i\notin \ocH_{\theta', \re(\xi e^{-i\theta'})}} \cF_{c_0}/ \cF_{c_i} = \bigoplus_{c_i <_{\theta'} \xi} \cF_{c_0}/ \cF_{c_i}.
\]
We conclude that $\Psi(\cF)$ is a co-Stokes structure of exponential type with exponents $\{c_1, \dots,c_n\}$.
\end{proof}

\begin{remark}\label{remark: stalk L at xi and theta}
For any closed subset $i \colon Z \hookrightarrow \bbC$, we denote $\cF(Z) \coloneqq i^*\cF(Z)$. Let $\cF$ be a constructible sheaf and $(\cL, \cL_<)\coloneqq\Psi(\cF)$. Then the stalk $\cL_{<\xi, \theta}$ is equal to $\cF
\big(\ocH_{\theta, \re(\xi e^{-i\theta})}\big)$. When $\xi$ is fixed and $\theta$ varies, different stalks $\cL_{<\xi, \theta}$ correspond to the spaces of sections of $\cF$ over different closed half-planes whose boundary contains $\xi$ (see \cref{fig: define stokes structure labeled by xi}).
\end{remark}

\begin{figure}[!ht]
	\centering
	\setlength{\unitlength}{0.4\textwidth}
	\begin{picture} (1,1)
		\put(0,0){\includegraphics[width=\unitlength]{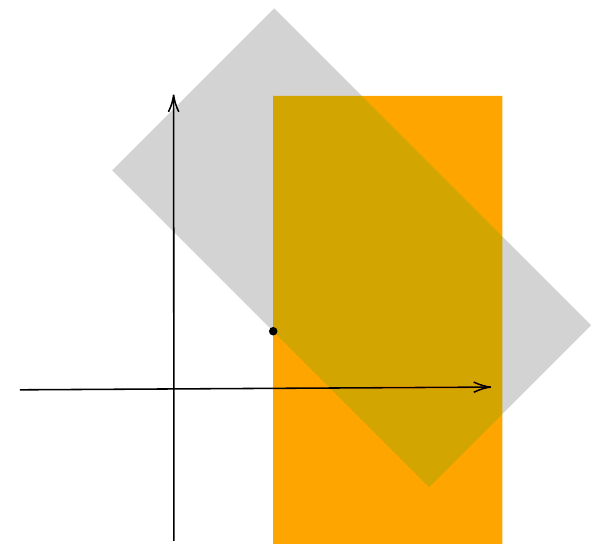}}
		\put(0.40,0.35){${\xi}$}
	\end{picture}
	\caption{Closed half-planes $\ocH_{\theta, \re(\xi e^{-i\theta})}$ for a fixed $\xi$ and varying $\theta$.}
	\label{fig: define stokes structure labeled by xi}
\end{figure}

\subsection{Topological Laplace transform to co-Stokes sheaves}\label{subsection: assemble Laplace transform}

Here we give a direct sheaf-theoretic description of the composite functor
\[\begin{tikzcd}[row sep = tiny]
    \Constrz \rar{\Psi} & \Stol \rar{\sigma} & \SSh\\ 
    \cF \rar[mapsto] & (\cL, \cL_<)  \rar[mapsto] & (\cL, \cG).
\end{tikzcd}\]
\begin{theorem}\label{theorem: describe sigma circ lt using pullback-pushforward}
For $z=e^\lambda e^{i\theta}\in\bbC^*$, $\theta \in [0,2\pi)$ and $\lambda\in\mathbb{R}$, write $\cH_z\coloneqq\cH_{\theta, \lambda}$. For $r \in \bbR$, consider the following diagram
\begin{equation*}
		\begin{tikzcd}
           & {\bigcup_{z\in\mathbb{C^*}}\cH_{z}\times\{z\}} \arrow[ld, "p_1"'] \arrow[rd, "q_1"] &              \\
\bbC &                                                                                                                                               & \bbC^*            \\ 
& {\bigcup_{z\in\mathbb{C^*}}\cH_{\theta,r}\times\{z\}} \arrow[lu, "{m_{r}}"'] \arrow[ru, "{n_{r}}"]
\end{tikzcd}
\end{equation*}
We have an isomorphism of co-Stokes sheaves
\[(\cL, \cG) \simeq \big(\colim_{r\rightarrow +\infty} R^0n_{r*}m_r^*\cF, \ R^0q_{1*}p_1^*\cF \big).\]

\end{theorem}

The proof follows from \cref{Lemma: composition for local system} and \cref{Proposition: composition for co-Stokes sheaf G}.

\begin{lemma} \label{Lemma: composition for local system}
We have an isomorphism of local systems
\[\cL\simeq  \colim_{{r \rightarrow +\infty}} R^0n_{r*}m_r^* \cF.\]
\end{lemma}
\begin{proof}
Given $r\in\bbR$, consider the following diagram.
\[\begin{tikzcd}
     & {\bigcup_{z\in \bbC^*}\cH_{\theta,r}}\times \{z\} \arrow[ld, "m_r"'] \arrow[rrd, "n_r"] \arrow[rd, "b"'] \arrow[dd, "a"'] &     &                         \\
\bbC &                                                                                                               & S^1 & \bbC^* \arrow[l, "\pi"] \\
     & {\bigcup_{\theta\in S^1}\cH_{\theta,r}}\times \{\theta\} \arrow[ru, "q_r"'] \arrow[lu, "p_r"]                                  &     &                        
\end{tikzcd}\]
The isomorphism follows from the following adjunctions \[\pi^*R^0q_{r*}p_r^*\cF \xrightarrow{\ \sim \ } \pi^*R^0q_{r*}R^0a_*a^*p_{r}^*\cF= \pi^*R^0\pi_* R^0n_{r*}m_r^*\cF\xleftarrow{\ \sim \ } R^0n_{r*}m_r^*\cF.\]
\end{proof}

By \cref{lemma: assemble Stokes filtratoin} and \cref{remark: stalk L at xi and theta}, the sheaf $\cG$ is the unique subsheaf of $\cL$ with $\cG_z=\cF(\ocH_z)\hookrightarrow \cL_z$.
Let $\cG'\coloneqq R^0q_{1*}p_1^*\cF$.
Note that $\cG$ and $\cG'$ have same stalks $\cF(\ocH_{z})$.
To identify them, it suffices to construct a monomorphism $\cG' \hookrightarrow \cL$.
We will first describe $\cG'$ as a colimit in \cref{Lemma: G' is a colimit}.

Let $r>0$ and ${D_r} = \{z\in \bbC^*: \lvert z\rvert< r\}$.
Let \[ X_{r} \coloneqq \Big(\bigcup_{z\in \oD_{r}}\cH_{z}\times \{z\}\Big)\cup \Big(\bigcup_{z\in \bbC^*\setminus \oD_{r}}{\cH_{\theta,r}\times \{z\}}\Big) \subset \bbC\times \bbC^*, \]
and denote the two projections by $s_r$ and $t_r$.
Let $\cG_{r}' \coloneqq R^0t_{r*}s_{r}^*\cF$. Let \[X\coloneqq{\bigcup_{z\in\mathbb{C^*}}\cH_{z}\times\{z\}}.\]

\begin{lemma}\label{Lemma: G' is a colimit}
We have an isomorphism between $\cG'$ and $ \colim_{r\rightarrow +\infty}\cG'_{r}$,
\end{lemma}

\begin{proof}
By definition, we have $\cH_z= \cH_{\theta,r}$ for $z\in \partial \oD_r$, and $\cH_{\theta,r}\subset \cH_z$ for $z\in \bbC^*\setminus \oD_r$.
So $X$ is an open subset of $X_r$.
Consider the following commutative diagram. 
\begin{equation*}
\begin{tikzcd}
\bbC &                                                         X_{r}      \arrow[l, "s_r"'] \arrow[r, "t_r"]                                                                                            & \bbC^*         \\
           &X \arrow[lu, "p_1"] \arrow[ru, "q_1"'] \arrow[u, hook, "j"]  &             
\end{tikzcd}
\end{equation*}
For any $\lambda\in\mathbb{R}$, there is a natural map $\cG'_r\rightarrow \cG'$ induced by the adjunction $\id\rightarrow R^0j_*j^*$. Furthermore, for any $0<r_1<r_2$, there is a natural inclusion of open sets $ X_{r_1}\hookrightarrow X_{r_2}$. This induces a natural map $\cG'_{r_2} \rightarrow\cG'_{r_1} $ by adjunction. Taking colimit as $r\rightarrow +\infty$ yields a sheaf morphism on $\bbC^*$ 
\[\psi\colon  \colim_{r\rightarrow +\infty}\cG'_{r}\longrightarrow \cG'.\]
Let us check that it induces an isomorphism on the stalks at each $z \in \bbC^*$.
Choose $r_z$ large enough so that $\oD_{r_z}$ contains $z$.
Note that taking stalk commutes with colimit, and the colimit $\colim_{r \to +\infty} G'_{r,z}$ stabilizes when $r > r_z$.
This completes the proof.
\end{proof}

\begin{proposition}\label{Proposition: composition for co-Stokes sheaf G}
There is a monomorphism $\cG' \hookrightarrow \cL'$, which induces an isomorphism between subsheaves $\cG$ and $\cG'$ of $\cL'$.   
\end{proposition}

\begin{proof}
Consider the following commutative diagram.
\begin{equation*}
\begin{tikzcd}
           & {\bigcup_{z\in\mathbb{C^*}}\cH_{\theta,r}\times\{z\}} \arrow[ld, "m_{r}"'] \arrow[rd, "n_{r}"] \arrow[dd, "j_{r}"]                                    &              \\
\bbC &                                                                                                                                                        & \bbC^* \\
           & X_{r} \arrow[lu, "s_{r}"] \arrow[ru, "t_{r}"']                                                                                                               &              \\
           &X= {\bigcup_{z\in\mathbb{C^*}}\cH_{\theta, \lambda_z}\times\{z\}} \arrow[luu, "p_1", bend left] \arrow[ruu, "q_1"', bend right] \arrow[u,"j"']  &             
\end{tikzcd}
\end{equation*}
By Lemmas \ref{Lemma: composition for local system} and \ref{Lemma: G' is a colimit},  we have \begin{align*}
     \cL' &= \colim_{r \rightarrow +\infty} R^0n_{r*}m_r^* \cF \\\cG' &=   \underset{r\rightarrow +\infty}{\colim} R^0t_{r*}s_r^*\cF.
\end{align*}
To construct a morphism $\phi$ between the two colimits
\[\phi\colon  \colim_{r\rightarrow +\infty}R^0t_{r*}s_r^*\cF\longrightarrow \colim_{r\rightarrow +\infty}R^0n_{r*}m_{r}^* \cF, \]
it suffices to construct a compatible family of maps $\phi_{r}\colon  R^0t_{r*}s_r^*\cF\rightarrow R^0n_{r*}m_{r}^* \cF$ for each $r>0$. Since $m_{r}=s_{r}\cir j_{r}$ and $n_{r}=t_{r}\cir j_{r}$, we have a natural morphism $ \phi_{r}\colon  R^0t_{r*}s_r^*\cF\rightarrow R^0n_{r*}m_{r}^* \cF$ induced by the adjunction $\id\rightarrow R^0j_{r*}j_r^*$. Furthermore, since each $\phi_r$ is a monomorphism, the colimit $\phi$ is also a monomorphism. So $\cG'$ is a subsheaf of $\cL'$ and it has the same stalks as $\cG$. Therefore we conclude the proof.
\end{proof}

\section{From co-Stokes structures to constructible sheaves}\label{section: From co-Stokes structures to constructible sheaves}
In this section, we will construct a topological Laplace transform functor $\Phi \colon \SSh \to \Constrz$ from co-Stokes sheaves to constructible sheaves with vanishing cohomology.

\begin{definition}\label{definition: inverse topological Laplace transform}
  Given a co-Stokes sheaf $(\cL, \cG)$, we define its \textsl{topological Laplace transform} $\Phi(\cL, \cG)$ as
  \begin{equation*}
    \Phi(\cL, \cG) \coloneqq Rp_{2!}q_2^!\cG,
  \end{equation*}
  where $ p_2, q_{2}$ are natural projections in the following diagram.
  \begin{equation*}
     \begin{tikzcd}
           & {\bigcup_{z\in \bbC^*}\cH_z\times\{z\}} \arrow[ld, "p_{1}"'] \arrow[rd, "q_{1}"] &              \\
\bbC &                                                                                                            & \bbC^* \\
           & \bigcup_{\xi \in \bbC}\{\xi\}\times S_{\xi} \arrow[lu, "p_2"] \arrow[ru, "q_2"']                  &          
  \end{tikzcd}
  \end{equation*}
\end{definition}

We will show that $\Phi(\cL, \cG)$ has only degree $0$ cohomology sheaf (\cref{lemma: inverse Laplace only  has H^0}), it is a constructible sheaf (\cref{Lemma: inverse Laplace is constructible sheaf}) and has vanishing cohomology on $\bbC$ (\cref{Theorem: inverse Laplace transform is in constr}).

\begin{remark}\label{remark: alternative description of unions of S1}
The subspace $\bigcup_{\xi \in \bbC}\{\xi\}\times S_{\xi}\subset \bbC\times\bbC^*$ can be viewed from two perspectives.
\begin{enumerate}[wide]
    \item \label{remark: C times S1} It is homeomorphic to $\bbC\times S^1$ via the map 
\begin{align*}
     h\colon  \bbC\times S^1&\longrightarrow \bigcup_{\xi \in \bbC}\{\xi \}\times S_{\xi}\\ 
     (\xi, \theta)&\longmapsto (\xi,e^{\re(\xi e^{-i\theta})}e^{i\theta}).
\end{align*}
The composition $p_2\cir h$ is equal to the projection $\bbC\times S^1\longrightarrow \bbC$. 
    \item \label{remark: alternative description of S1 (2)} It is also homeomorphic to $\mathbb{R}\times \bbC^*$ via the map
    \begin{align*}
        h'\colon  \mathbb{R}\times \bbC^*&\longrightarrow \bigcup_{\xi \in \bbC}\{\xi \}\times S_{\xi}\\
        (a,z=e^re^{i\theta})&\longmapsto  (re^{i\theta}+ae^{i(\theta+\pi/2)},z) .
    \end{align*}
    The composition $q_2\cir h'$ is the projection $\bbR\times \bbC^*\longrightarrow \bbC^*$.
\end{enumerate}
\end{remark}
\begin{remark}\label{q2 pullback's global section}
One can prove that if $f:X\rightarrow Y$ is a continuous open map with connected fibers, and if $X$ is connected, then for any sheaf $\cG$ on $Y$, we have $f^*\cG(X)=\cG(Y)$. In particular, this holds for $q_2$ by \cref{remark: alternative description of unions of S1}(\ref{remark: alternative description of S1 (2)}).

\end{remark}

In order to prove that $\Phi(\cL, \cG)$ has only degree 0 cohomology sheaf, the following lemma about co-Stokes structure is crucial.

\begin{lemma}\label{lemma: only S1 survives on circle}
Let $(\cL,\cL_<)$ be a co-Stokes structure.
We have $H^i(S^1,\cL_{<\xi}) = 0$ for all $i\neq 1$.
\end{lemma} 

\begin{proof}
By \cite[Proposition 3.2.2]{Sheaves_on_Manifolds}, we have  $H^i(S^1, \cL_{<\xi})=0$ for $i\geq 2$.
Furthermore, $H^0(S^1, \cL_{<\xi})=0$ by \cite[Corollary 3.16]{Sabbah_introduction_to_Stokes_structure} (see also \cite{Malgrange1982_la_classification_des_connexions_irregulieres_a_une_variable}).
\end{proof}

\begin{lemma}\label{lemma: inverse Laplace only  has H^0}
The complex $ \Phi(\cL, \cG)=Rp_{2!}q_2^!\cG$  has only degree 0 cohomology sheaf and $\cH^0(Rp_{2!}q_2^!\cG)=R^1p_{2*}q_2^*\cG$.
\end{lemma}

\begin{proof}
By \cref{remark: alternative description of unions of S1}(\ref{remark: alternative description of S1 (2)}), the composition $q_2\cir h$ is a topological fibration of relative dimension $1$. It follows from \cite[Theorem 3.2.17]{Dimca_Sheaves_in_Topology} that $q_2^!=q_2^*[1]$. Since $p_2$ is proper, we have $Rp_{2!}=Rp_{2*}$. Then $\cH^i(Rp_{2!}q_2^!\cG) = \cH^{i+1}(Rp_{2*}q_{2}^*\cG)$. Since a sheaf of vector spaces is injective if and only if it is flasque (see \cite[lemma 1.13]{Borel_sheaf_theoretic_intersection_cohomology}), \cref{q2 pullback's global section} implies that $\cH^{i+1}(Rp_{2*}q_{2}^*\cG) = R^{i+1}p_{2*}q_{2}^*\cG$. Since $p_2$ is proper, by the proper base change theorem, at any $\xi\in\bbC$, we have 
\[(R^ip_{2*}q_2^*\cG)_\xi \simeq H^i(S_\xi, \cG|_{S_\xi}) \simeq H^i(S^1, \cL_{<\xi}),\]
where $\sigma(\cL,\cL_<)=(\cL,\cG)$.
By \cref{lemma: only S1 survives on circle}, for $i\neq 1$, we have $H^i(S^1, \cL_{<\xi})=0$.
We conclude that  $\Phi(\cL, \cG)$ has only degree 0 cohomology sheaf whose stalk at $\xi$ is $H^1(S_\xi,\cG)$.
\end{proof}
We compute the following cohomology for later use.

\begin{lemma}\label{lemma: calculation of cohomology}
Let $D = \{ \xi \in \bbC : \abs{\xi < 1}\}$ denote the open unit disk, and $\oD$ the closed unit disk.
Consider the inclusions $i\colon D \cap \{\re(\xi)>0\} \hookrightarrow D$, $j\colon \oD \cap \{\re(\xi)>0\} \hookrightarrow \oD$, and $k\colon D \hookrightarrow \oD$.
The following hold.
\begin{enumerate}[wide]
    \item $H^*(D,i_!\bbC)=0$.
    \item $H^*(\oD,j_!\bbC)=0$.
    \item $H^2(\oD,k_!\bbC)=\bbC$ and $H^q(\oD,k_!\bbC)=0$ for $q\neq 2$. 
\end{enumerate}
\end{lemma}

\begin{proof}
For (1), denote $l\colon  D\setminus U\hookrightarrow D$. It follows from the exact sequence of sheaves \[0\longrightarrow i_!i^*\bbC\longrightarrow \bbC\longrightarrow l_*l^*\bbC\longrightarrow 0.\]
Same for (2) and (3).
\end{proof}

\begin{lemma}\label{lemma: small open cover are Leray cover}
Let $(\cL,\cL_<)$ be a co-Stokes structure with exponents $\{c_1,\dots, c_n\}$. Let $\cU = \{U_1, \dots, U_m\}$ be an open cover of $S^1$ consisting of small open intervals. There is an isomorphism between the Čech cohomology group and sheaf cohomology group
\[H^q(\cU,\cL_{<\xi}) = H^q(S^1,\cL_{<\xi}).
\]
\end{lemma}

\begin{proof}
Let $I$ be any finite intersections of $U_i$'s. Then $I$ is a small open interval. 
By \cref{proposition: Stokes filtration trivializes over good open intervals}, there exists an isomorphism \[\cL|_I \xrightarrow{\ \sim \ } \bigoplus_{c_i} \gr_{c_i} \cL|_I,\] which is compatible with filtration. 
Denote $j_i\colon I_i=\{\theta\in I|c_i<_\theta \xi\}\hookrightarrow I$ and $k_i$ its complement. Then we have \[{\cL_{<\xi}}|_I \xrightarrow{\ \sim \ }\bigoplus_{c_i} j_{i!}j_i^*(\gr_{c_i}\cL|_I).\] 
If $I_i=I$, then $j_i=\Id$ and $H^p(I,\gr_{c_i}\cL|_I)=0$ for $p\geq 1.$ If $I_i\neq I$, one can prove that $H^p(I,j_{i!}j_i^*(\gr_{c_i}\cL|_I))=0$ for $p \geq 1$, using the exact sequence of sheaves 
\[0\longrightarrow j_{i!}j_i^*(\gr_{c_i}\cL|_I)\longrightarrow(\gr_{c_i}\cL|_I) \longrightarrow k_{i*} k_i^*(\gr_{c_i}\cL|_I)\longrightarrow 0.\]
Then for any $\xi \in \bbC$, $\cL_{<\xi}$ is acyclic on every finite intersection of elements of $\cU$. So, we have an isomorphism \[H^q(\cU,\cL_{<\xi}) = H^q(S^1,\cL_{<\xi}).
\]
\end{proof}

\begin{definition}\label{definition: small and good sectors}
Let $(\cL, \cG)$ be a co-Stokes sheaf with exponents $\{c_1, \dots,c_n\}$.
An open sector $U\subset \bbC^*$ is called \emph{good} (resp.\ \emph{small}) if its projection to $S^1$ is a good (resp.\ \emph{small}) open interval (see \cref{definition: Stokes direction and small}).
\end{definition}

\begin{lemma}\label{cohomology of G over open sectors are 0}
Let $(\cL, \cG)$ be a co-Stokes sheaf with exponents $\{c_1, \dots,c_n\}$. If $U$ is a small open sector, then $H^*(U, \cG)=0$. 
\end{lemma}
\begin{proof}
For each $c_i$, let $D_{c_i}$ be as in \cref{notation: circles}, and denote by $j_i\colon (\bbC^* \setminus \oD_{c_i}) \cap U \hookrightarrow U$ the inclusion.
By \cref{definition: co-Stokes sheaf}, we have a decomposition
\begin{equation}\label{eq: decomposition of G over good sector}
    \cG|_U = \bigoplus_{i=1}^n j_{i!} \cV_i,
\end{equation}
where $\cV_i$ is a constant sheaf over $(\bbC^*\setminus\oD_{c_i}) \cap U$.
We conclude the proof by \cref{lemma: calculation of cohomology}(1).
\end{proof}

\begin{lemma}\label{compute F over open set}
Let $U\subset \bbC$ be an open subset, then $\cF(U)=H^1\big(\bigcup_{\xi\in U}S_\xi,\cG\big)$.
\end{lemma}
\begin{proof}
Since $p_2$ is proper and $q_2$ is a topological fibration of relative dimension 1, the restriction of $\cF$ to $U$ is
\begin{equation}\label{eq: computation 1}
    j^*Rp_{2!}q_2^!\cG \simeq Rp'_{2!}j^*q_2^!\cG =Rp'_{2*}j'^*q_2^*\cG[1]= \cH^1(Rp_{2*}'q_{2}'^*\cG),
\end{equation}
where $p_2',q_2'$ are the projections as in the following diagram.
\begin{equation*}
     \begin{tikzcd}
           U & & \bbC^* \\
           & \underset{\xi \in U}{\bigcup}\{\xi\}\times S_{\xi} \arrow[lu, "p_2'"] \arrow[ru, "q_2'"']                  &          
\end{tikzcd}
\end{equation*}
Let $I^\bullet$ be an injective resolution of $\cG$. Similar to the proof of \cref{lemma: inverse Laplace only  has H^0}, we can show that $q_2'^*I^\bullet$, as the restriction of $q_2^*\cG$ to an open set, is an injective resolution of $q_2'^*\cG$. So we have
\[\cH^1(Rp_{2*}'q_{2}'^*\cG)= R^1p_{2*}'q_{2}'^*\cG .\]
Let $\pi: U\rightarrow *$ be the constant map, then it follows from \eqref{eq: computation 1} that $\cF(U) = R^0\pi_* (R^1p_{2*}'q_2'^*\cG)$. 
To compute this, we use the following exact sequence from the Grothendieck spectral sequence,
\[0\rightarrow R^1\pi_*(R^0p_{2*}'q_2'^*\cG)\rightarrow R^1(\pi_*p_{2*}')(q_2'^*\cG)\rightarrow R^0\pi_*(R^1p_{2*}'q_2'^*\cG)\rightarrow R^2\pi_*(R^0p_{2*}'q_2'^*\cG)\rightarrow.\]
Since $R^0p_{2*}'q_2'^*\cG=0$, we have an isomorphism $R^1(\pi_*p_{2*}')(q_2'^*\cG)\xrightarrow{\sim} R^0\pi_*(R^1p_{2*}'q_2'^*\cG)$, which implies that 
\[\cF(U) = H^1\bigg(\bigcup_{\xi\in U}\{\xi\}\times S_\xi,q_2'^*\cG\bigg).\] 
By
\cref{remark: alternative description of unions of S1}(2), we have a homotopy equivalence between $\bigcup_{\xi \in U}\{\xi \}\times S_\xi$ and $\bigcup_{\xi\in U}S_\xi$. Using \cite[Corollary 2.7.7(ii)]{Sheaves_on_Manifolds} and the Grothendieck spectral sequence, we have \[H^1\bigg(\bigcup_{\xi\in U}\{\xi\}\times S_\xi,q_2'^*\cG\bigg)= H^1\bigg(\bigcup_{\xi\in U} S_\xi,\cG\bigg).\] This concludes the proof. \end{proof}

\begin{lemma}\label{Lemma: inverse Laplace is constructible sheaf}
The topological Laplace transform $\Phi(\cL, \cG)$ is a constructible sheaf on $\bbC$ with singularities $\{c_1, \dots,c_n\}$.
\end{lemma}
\begin{proof}
Let $c_0\in \bbC \setminus \{c_1, \dots, c_n\}$ and $j\colon D \hookrightarrow \bbC\setminus\{c_1, \dots,c_n\}$ be a small open disk centered at $c_0$. Let us show that $\cF|_D$ is a constant sheaf. As in the proof of \cref{compute F over open set}, the restriction of $\Phi(\cL,\cG)$ to $D$ is $R^1p_{2*}'q_2'^*\cG$,
where $p_2'$ and $q_2'$ are the projections as in the following diagram.
\begin{equation*}
     \begin{tikzcd}
           \bbC &                                                                                                            & \bbC^* \\
           & \underset{\xi \in D}{\bigcup}\{\xi\}\times S_{\xi} \arrow[lu, "p_2'"] \arrow[ru, "q_2'"']                  &          
\end{tikzcd}
\end{equation*}
This is the sheafication of the presheaf sending $V\subset D$ to $H^1(\cup_{\xi\in V}S_{\xi}, \cG)$. 
To show that it is a constant sheaf, it suffices to show that for inclusions $V_1\subset V_2\subset D$, the restriction map \[r\colon H^1(\cup_{\xi\in V_1}S_{\xi}, \cG)\longrightarrow H^1(\cup_{\xi\in V_2}S_{\xi}, \cG)\] is an isomorphism.
Let us take a Čech cover of $\bbC^*$ consisting of two good open sectors $U_1$ and $U_2$.
Then the restrictions of $\cG$ to $U_1$ and $U_2$ are isomorphic to direct sums of extensions by zeros of constant sheaves as in \eqref{eq: decomposition of G over good sector}.
Since $D$ is disjoint from $\{c_1, \dots, c_n\}$, we conclude that $r$ is an isomorphism by the homotopy invariance of singular cohomology and the long exact sequence associated to an open embedding and its closed complement.
\end{proof}

Having proved that $\Phi(\cL, \cB)$ is constructible, next we show that it has vanishing cohomology.

\begin{lemma}\label{RGamma=0 condition}
Let $\cF$ be a constructible sheaf on $\bbC$ with singularities $\{c_1, \dots,c_n\}$, and let $c_0$ be a nonsingular point. 
Then $H^*(\bbC, \cF)=0$ if and only if the following two conditions hold.
    \begin{enumerate}[wide]
     \item The map $\cF\longrightarrow R^0j_*j^*\cF$ is injective, where  $j\colon  \bbC \setminus \{c_1, \dots,c_n\}\hookrightarrow \bbC$.
     \item $\dim \cF_{c_0} = \sum_{1\leq i\leq n} (\dim \cF_{c_0}-\dim \cF_{c_i})$.
\end{enumerate}
\end{lemma}
\begin{proof}
    We refer to the proof of \cite[Theorem 2.29]{KKP_Hodge_theoretic_aspect_of_mirror_symmetry}.
\end{proof}

Now we verify the conditions (1) and (2) above for $\cF=\Phi(\cL, \cG)$ in \cref{proposition: RGamma=0 condition (2)} and \cref{lemma: the adjunction map is injective} 

\begin{proposition}\label{proposition: RGamma=0 condition (2)}
Given a co-Stokes structure $(\cL, \cL_{<})$ with exponents $\{c_1, \dots,c_n\}$ and $c_0\in \bbC\setminus\{c_1, \dots,c_n\}$, we have 
\begin{align*} 
   \dim H^1(S^1, \cL_{<c_0})-\dim H^1(S^1, \cL_{<c_i}) &= \dim\gr_{c_i}\cL\\
    \sum_{1\leq i\leq n} \big(\dim H^1(S^1, \cL_{<c_0})-\dim H^1(S^1, \cL_{<c_i})\big) &= \dim H^1(S^1, \cL_{<c_0}).
    \end{align*}    
\end{proposition}
\begin{proof}
By \cref{Lemma: inverse Laplace is constructible sheaf}, for each $\xi \in \bbC\setminus\{c_1,\dots,c_n\}$, the dimension of $H^1(S^1,\cL_{<\xi})$ is constant.
So for a singular point $c_i, $ we may assume that $c_0$ is a nonsingular point near it. Let $\theta_1$ and $\theta_2$ be the directions where $S_{c_i}$ intersects $S_{c_0}$. Let $\epsilon$ be small enough such that for all $\theta \in (\theta_1-\epsilon,\theta_1+\epsilon)\setminus\theta_1$ the circle $S^1_{c_j}$ intersects neither $S^1_{c_i}$ nor $S^1_{c_0}$.
Let $U_1 \coloneqq (\theta_1-\epsilon,\theta_1+\epsilon)$, $U_2\coloneqq (\theta_1,\theta_1+\pi)$, $U_3\coloneqq (\theta_1+\pi-\epsilon,\theta_1+\pi+\epsilon)$, and $U_4\coloneqq (\theta_1+\pi,\theta_1+2\pi)$.
All of them are small open intervals as in \cref{definition: Stokes direction and small}.
Consider the open cover $\cU=\{U_1,U_2,U_3,U_4\}$ of $S^1$.
Then the first equality follows from \cref{lemma: small open cover are Leray cover} and the Čech cohomology of $\cU$. 

For the second equality, we choose $c_0$ to be a point near infinity. By a similar Čech cohomology computation, one can show that \[\dim H^1(S^1, \cL_{<c_0})=\sum_{1\leq i\leq n} \dim \gr_{c_i}\cL.\] Then second equality follows from the first equality. 
\end{proof}

\begin{lemma}\label{cohomology of G over C^* vanish}
     Let $(\cL, \cG)$ a co-Stokes sheaf with exponents $\{c_1, \dots,c_n\}$. We have \begin{equation*}
         H^*(\bbC^*, \cG)=0.
     \end{equation*}
\end{lemma}

\begin{proof}
Let $\cU=\{U_1, \dots,U_n\}$ be an open cover of $\bbC^*$ consisting of small open sectors. Since any finite intersection of small open sectors is still a small open sector, by \cref{cohomology of G over open sectors are 0}, for any $m$, \[H^0(U_1\cap\dots\cap U_m, \cG|_{U_1\cap\dots \cap U_m})=0.\]
This implies that $H^k(\cU,\cG)=0$. 
Also by \cref{cohomology of G over open sectors are 0}, for any $k$ 
\[H^k(\bbC^*,\cG)=H^k(\cU,\cG)=0.\] This concludes the proof.
\end{proof}

\begin{lemma}\label{lemma: the adjunction map is injective}
 The canonical map $\cF\longrightarrow R^0j_*j^*\cF$ is injective.
\end{lemma}

\begin{proof}
The fiber sequence
\begin{equation*}
i_!i^!\cF\longrightarrow\cF\longrightarrow Rj_*j^*\cF
\end{equation*} 
induces a long exact sequence on cohomology sheaves \begin{equation*}
0 \longrightarrow \cH^0(i_!i^!\cF)\longrightarrow \cF\longrightarrow R^0j_*j^*\cF\longrightarrow \cH^1(i_!i^!\cF)\longrightarrow 0.
\end{equation*}
To prove that $\cF\longrightarrow R^0j_*j^*\cF $ is injective, it suffices to prove that $\cH^0(i_!i^!\cF)=0$. Consider the following diagram, where $p_2$ is proper. 

\begin{equation*}
   \begin{tikzcd}
\{c_i\}\times S_{c_i}  \arrow[d, "{p_2}'"'] \arrow[r, "{i}'"'] & \bigcup_{\xi \in \bbC}\{\xi\}\times S_{\xi}  \arrow[d, "p_2"] \\
\{c_i\} \arrow[r, "i"]                                                         & \bbC                                                                     
\end{tikzcd}
\end{equation*}

Let $\tilde{i}=q_2\cir i'\colon  \{c_i\}\times S_{c_i}\longrightarrow \bbC^*$. Since $Rp_{2!}=Rp_{2*}$, by \cite[Theorem 3.2.13]{Dimca_Sheaves_in_Topology} we have
\begin{equation*}
i_!i^!\cF = i_*i^!Rp_{2!}q_2^!\cG \simeq i_*R{p}_{2!}'{i}'^!q_2^!\cG   =i_*R{p}_{2*}'\tilde{i}^!\cG,  
\end{equation*}
Thus 
\begin{equation*}
   \cH^0(i_! i^!\cF)= \cH^0(i_*R{p}_{2*}'\tilde{i}^! \cG)=i_* \cH^0(Rp_{2*}'\tilde{i}^!\cG) = i_*\mathbb{H}^0(S_{c_i}, \tilde{i}^!\cG)
\end{equation*} 
where $\mathbb{H}^0(S_{c_i}, \tilde{i}^!\cG)$ denote the 0-th hypercohomology group.
We have the following exact sequence of cohomology groups (see \cite[Page 45]{Dimca_Sheaves_in_Topology})
\begin{equation*}
0\longrightarrow \mathbb{H}^0(S_{c_i}, \tilde{i}^!\cG) \longrightarrow H^0(\bbC^*, \cG) \longrightarrow H^0(\bbC^* \setminus S_{c_i}, \cG)
\end{equation*}
By \cref{cohomology of G over C^* vanish}, we have $H^0(\bbC^*, \cG)=0$.
Therefore, $\mathbb{H}^0(S_{c_i}, \tilde{i}^!\cG)=0$, then $\cH^0(i_! i^!\cF)=0$, completing the proof.
\end{proof}

\begin{theorem}\label{Theorem: inverse Laplace transform is in constr}
Let  $(\cL, \cG)$ be a co-Stokes sheaf with exponents $\{c_1, \dots,c_n\}$, and let $\cF= \Phi(\cL, \cG) $. We have \begin{equation*}
 R\Gamma(\cF)=0.
\end{equation*}
\end{theorem}
\begin{proof}
This follows from \cref{RGamma=0 condition}, \cref{proposition: RGamma=0 condition (2)}, and \cref{lemma: the adjunction map is injective}.
\end{proof}

\section{The topological Laplace transform is an equivalence}

In this section, we prove that topological Laplace transform functors we constructed in Sections \ref{section: From constructible sheaves to co-Stokes structures} and \ref{section: From co-Stokes structures to constructible sheaves} 

\[\begin{tikzcd}
    & \SSh \arrow[rd, "\Phi"] & \\
    \Stol \arrow[ru, "\sigma"] & & \Constrz \arrow[ll, "\Psi"]
 \end{tikzcd}\]
give an equivalence of categories.

\begin{theorem} \label{theorem: const0 and Stokes are equivalent 6.1}
    We have an equivalence of categories 
    \begin{equation*}
        \Phi\cir \sigma \colon  \Stol \mathrel{\substack{\xrightarrow{\hspace{2em}} \\[-0.6ex] \xleftarrow{\hspace{2em}}}}  \Constrz  \colon \Psi.
    \end{equation*}
\end{theorem}

We will construct two natural isomorphisms, $(\Phi \cir \sigma) \cir \Psi\xrightarrow{\ \sim \ } \id$ in \cref{theorem: it sigma lt=id} and $\Psi \cir (\Phi \cir  \sigma)\xrightarrow{\ \sim \ } \id$ in \cref{theorem: lt it sigma = id}.

Consider the following diagram 
\begin{equation*}
    \begin{tikzcd}
\bbC & \bigcup_{z\in \bbC^*}\{\xi\}\times\{z\} \arrow[l, "p_3"'] \arrow[r, "q_3"] & \bbC^*,
\end{tikzcd}
\end{equation*}
where $p_3$ and $q_3$ are the two projections, and $\xi$ is the center of $\hbarz$ (see \cref{notation: halfplanes}). For $\xi \neq 0$, its preimage $(q_3\cir p_3^{-1})(\xi)$ consists of exactly two points, which we denote as $z_1$ and $z_2$. For $\xi=0$, its preimage $(q_3\cir p_3^{-1})(0)$ is the unit circle $S_{\xi=0}$.

Our construction of the natural isomorphism $\Psi \cir (\Phi \cir  \sigma)\xrightarrow{\ \sim \ } \id$ is motivated by the observation that, given $\cF\in \Constrz$ and $\xi\in \bbC$, we have
\[    \dim \ocH_{z_1} +\dim\ocH_{z_2} = \dim \cF_\xi,
\]
over a dense open subset of $\bbC$ (see Lemmas \ref{Lemma: F at a point = F at two closed half-planes} and \ref{lemma: co stokes structure and decomposition H1}). 
\begin{definition}
Let $\cF\in \Constrz$ with singularities $c_1, \dots,c_n$. 
We call $\xi \in\bbC$ a \emph{Stokes point}, if $\xi=0$, or $c_i$, or for some $c_i\neq \xi$, the vector $\xi - c_i$ is orthogonal to $\xi$.
\end{definition}

\begin{lemma}
\label{Lemma: F at a point = F at two closed half-planes}
Let $\cF \in \Constrz$ with singularities $\{c_1,\dots,c_n\}$.
If $\xi\neq 0$ is not a Stokes point, then
\begin{equation*}
        \dim\cF_\xi  = \dim \cF\big(\ocH_{z_1}\big)+\dim \cF\big(\ocH_{z_2}\big)
    \end{equation*} 
where $(q_3\cir p_3^{-1})(\xi)=\{z_1,z_2\}$.
\end{lemma}
\begin{proof}
Let $c_0$ be a nonsingular point. We choose Gabrielov paths from $c_i$'s to $c_0$, which identify each $\cF_{c_i}$ as a subspace of $\cF_{c_0}$ (see \cite[Theorem 2.29]{KKP_Hodge_theoretic_aspect_of_mirror_symmetry}). 
Then one can show that \[\dim \cF_{\xi}=\dim \cF_{c_0}=\sum_{i=1}^n \dim \cF_{c_0}/\cF_{c_i}.\] 
Since $\xi$ is not a Stokes point, each singular point $c_i$ lies on either $\ocH_{z_1}$ or $\ocH_{z_2}$. Assume the singularities $\{c_{i_1},\dots,c_{i_p}\}\subset \ocH_{z_1}$ and $\{c_{j_1},\dots,c_{j_q}\}\subset \ocH_{z_2}$ with $p+q=n$. Then the lemma follows from the following equalities
\[\dim (\ocH_{z_1}) = \sum_{k=1}^q \dim \cF_{c_0}/\cF_{c_{j_k}} \quad \text{and} \quad\dim (\ocH_{z_2}) = \sum_{k=1}^p \dim \cF_{c_0}/\cF_{c_{i_k}} .\]
\end{proof}
Similarly, for all $\xi \in \bbC$, let $z_1, z_2\in\bbC^*$ such that $\ocH_{z_i}$ is centered at $\xi$ and $\ocH_{z_1}\cap \ocH_{z_2}=l_{z_1}=l_{z_2}$. Then we have \[\dim \cF(\ocH_{z_1})+\dim \cF(\ocH_{z_2})=\dim \cF(l_{z_1}).\]

\begin{definition}\label{definition: stokes point of sheaf sheaf}
Let $(\cL, \cG)$ be a co-Stokes sheaf with exponents $\{c_1, \dots,c_n\}$. 
We call a point $z\in \bbC^*$ a \emph{Stokes point}, if $z$ lies on some $S_{c_i}$.
\end{definition}
\begin{remark}\label{stokes point of sheaves and co-stokes structure}
Let $(\cL, \cG)$ be a co-Stokes sheaf with exponents $\{c_1, \dots, c_n\}$. 
Let $\cF\in \Constrz$ with singularities $\{c_1, \dots, c_n\}$.
Let $z\in \bbC^*$ and $\xi$ be the center of $\hbarz$. 
If $z=e^\lambda e^{i\theta}$ lies on $S_{c_i}$, then $e^\lambda = e^{\re(c_i e^{-i\theta})}$ by \cref{notation: circles}. 
Consequently $c_i$ lies on the boundary of $\ocH_z$ by \cref{notation: halfplanes}.
Thus
\begin{enumerate}[wide]
    \item the point $z$ is a Stokes point of $(\cL, \cG)$ if and only if $\xi$ is a Stokes point of $\Phi(\mathcal{L, \cG})$,
    \item the point $\xi$ is a Stokes point of $\cF$ if and only if $z$ is a Stokes point of $\Psi(\cF)$.
\end{enumerate}
\end{remark}
Similar to \cref{Lemma: F at a point = F at two closed half-planes} for constructible sheaves, we have the following equalities for co-Stokes structures and co-Stokes sheaves.
\begin{lemma}\label{lemma: co stokes structure and decomposition H1}
Let $(\cL,\cL_<)$ be a co-Stokes structure with exponents $\{c_1,\dots,c_n\}$. 
For any $\xi\in\bbC$, let $\theta$ be a direction where none of $S_{c_i}$, with $c_i\neq \xi$, intersects $S_\xi$. 
Then, we have \[        \dim H^1(S^1, \cL_{<\xi}) = \dim \cL_{<\xi, \theta}+\dim \cL_{<\xi, \theta+\pi}.
\]

\end{lemma}

\begin{proof}
Let $U_1\coloneqq (\theta-\epsilon, \theta+\epsilon)$ be a small open interval such that at any $\theta'\in U_1$, none of $S_{c_i}$ with $\xi\neq c_i$ intersects $S_\xi$. 
By \cref{proposition: Stokes filtration trivializes over good open intervals}, $\cL_{<\xi}|_{U_1}$ is a constant sheaf and $H^0(U_1,\cL_{<\xi}|_{U_1})=\cL_{<\xi, \theta}$.
Let $U_3 \coloneqq (\theta+\pi-\epsilon,\theta+\pi+\epsilon)$. Similarly, $\cL_{<\xi}|_{U_3}$ is also a constant sheaf and $H^0(U_3,\cL_{<\xi}|_{U_3})=\cL_{<\xi,\theta+\pi}$.
Let $U_2\coloneqq (\theta,\theta+\pi)$ and $U_4\coloneqq (\theta+\pi,\theta+2\pi)$. Then $U_2$ and $U_4$ are small open intervals on which $S_\xi$ intersects each $S_{c_i}$, with $c_i\neq \xi$, exactly once. 
Then for $i=3$ or $4$, $H^0(U_i,\cL_{<\xi}|_{U_i})=0$. Using the Čech complex of $\{U_1,U_2,U_3,U_4\}$, we obtain 
\begin{equation*}
                \dim H^1(S^1, \cL_{<\xi}) = \dim \cL_{<\xi, \theta}+\dim \cL_{<\xi, \theta+\pi}.
\end{equation*}
\end{proof}

\begin{lemma}\label{lemma: sheaf cohomology of Lz and points}
Let $(\cL, \cG)$ be a co-Stokes sheaf with exponents $\{c_1, \dots,c_n\}$. Assume $z=e^re^{i\theta}\in \bbC^*$ is not a Stokes point of $(\cL,\cG)$.  Let $\xi$ be the center of $\ocH_z$ and $z'= e^{-r} e^{i(\theta+\pi)}$. 
Then, we have \begin{equation*}
        \dim H^1(S_\xi, \cG|_{S_{\xi}} ) = \dim \cG_{z}+\dim \cG_{z'}. 
\end{equation*}
\end{lemma}

\begin{proof}
Since $z$ is not a Stokes point of $(\cL,\cG)$, by \cref{stokes point of sheaves and co-stokes structure}, $\xi$ is not a Stokes point of $\Phi(\cL,\cG)$, which implies that $S_\xi$ does not intersect any $S_{c_i}$ at $z$. This allows us to apply \cref{lemma: co stokes structure and decomposition H1} to conclude our proof.
\end{proof}

\subsection{Natural isomorphism $(\Phi\cir \sigma) \cir \Psi\xrightarrow{\ \sim \ } \id$}
In this section, we prove the following \cref{theorem: it sigma lt=id}. As a corollary, $\Phi$ is compatible with Fourier transform of $D$-modules (see \cref{corollary: inverse Laplace transform is compatible}).

\begin{theorem}\label{theorem: it sigma lt=id}
Consider the following diagram. 
\begin{equation*}
\begin{tikzcd}
           & U= {\bigcup_{z\in \bbC^*}\cH_{z}\times \{z\}} \arrow[rd, "q_{1}"] \arrow[ld, "p_{1}"'] &              \\
\bbC &                                                                                             & \mathbb{C^*} \\
           & V=\bigcup_{\xi \in \bbC}\{\xi \}\times S_{\xi}  \arrow[lu, "p_2"] \arrow[ru, "q_2"']                 &             
\end{tikzcd}
\end{equation*}
Given $\cF\in\Constrz$, there is a natural isomorphism \begin{equation*}
    (\Phi\cir \sigma) \cir \Psi(\cF)\xrightarrow{\ \sim \ }\cF.
\end{equation*}
\end{theorem}

\begin{corollary} \label{corollary: inverse Laplace transform is compatible}
The topological Laplace transform $\Phi$ is compatible with Fourier transform of $D$-modules via Riemann-Hilbert correspondence, i.e.\ the following diagram commutes.
\begin{equation*}
    \begin{tikzcd}
        \Conn \arrow[d,"\RH", "\sim"'] \arrow[rr, "\Theta"]& & \operatorname{Hol_{rs}^0}(\bbA^1) \arrow[d, "\RH", "\sim"']\\ \Stol \arrow[rr,"(\Phi\cir \sigma)\otimes \bbC"]\otimes_\bbQ\bbC & & \Constrz\otimes_\bbQ \bbC 
    \end{tikzcd}
\end{equation*}
Here, 
\begin{enumerate}[wide]
    \item $\mathrm{Hol_{rs}^0}$ is the category of regular holonomic algebraic $D$-modules $M$ on $\bbA^1$ with vanishing de Rham cohomology.
    \item $\Conn$ is the category of pairs $(\cM, \nabla)$ where $\cM$ is a finite dimensional vector space over $\bbC\{u\}[u^{-1}]$ and $\nabla$ is an irregular meromorphic connection of exponential type. 
    \item $\RH$ denotes the (irregular) Riemann-Hilbert correspondence (see \cite{Malgrange1982_la_classification_des_connexions_irregulieres_a_une_variable} and \cite[\S II]{Sabbah_isomonodromic_deformations_and_frobenius_manifolds}).  
    \item $\Theta$ denotes the composition of Fourier transform, the pullback by $j:\bbA^1\setminus 0\hookrightarrow \bbA$, and the pushfoward by $i\colon u\mapsto 1/u$,
\[
\Theta \coloneqq  \FT\cir j_*\cir i_*.
\]
\end{enumerate}
\end{corollary}

\begin{proof}
It follows from \cite[\S 7]{Sabbah_introduction_to_Stokes_structure}  that the following diagram commutes.
\begin{equation*}
    \begin{tikzcd}
        \Conn \arrow[d,"\RH", "\sim"'] & & \operatorname{Hol_{rs}^0}(\bbA^1) \arrow[d, "\RH", "\sim"'] \arrow[ll, "\Theta^{-1}"']\\ \Stol \otimes_\bbQ\bbC & & \Constrz\otimes_\bbQ \bbC \arrow[ll,"\Psi\otimes \bbC"']
    \end{tikzcd}
\end{equation*}
Now the corollary follows from \cref{theorem: it sigma lt=id}.
\end{proof}

\begin{proposition}\label{gluing U V to Z}
Consider the open subset $U$ and the closed subset $V$ in $\bbC\times \bbC^*$ defined as
\begin{equation*}
    U \coloneqq {\bigcup_{z\in \bbC^*}\cH_{z}\times \{z\}}\quad \text{and} \quad V \coloneqq \bigcup_{\xi \in \bbC}\{\xi \}\times S_{\xi}.
\end{equation*}
Then $Z \coloneqq U\cup V$ is a closed subset $Z$ in $\bbC\times \bbC^*$.
Its interior is $U$, and its boundary is $V$.
\end{proposition}

\begin{proof}
Let $z=e^\lambda e^{i\theta}\in \bbC^*$ for $\lambda\in\bbR$ and $\theta\in[0,2\pi)$. A point $(\xi, z)$ belongs to $  \ocH_{\theta, \lambda}\times \{z\} $ if and only if $\re(\xi e^{-i\theta})\geq \lambda$.
Let $D_\xi$ and $S_\xi$ be as in \cref{notation: circles}.
Then a point $(\xi, z) \in
\{\xi\}\times \oD_{\xi}  $ if and only if $\re(\xi e^{-i\theta})\geq \lambda$.
So we can rewrite $Z$ as 
\begin{equation*}
    Z= \bigcup_{\xi \in \bbC}\{\xi\}\times \oD_{\xi}^* = {\bigcup_{z\in \bbC^*}\ocH_{\theta, \lambda}\times \{z\}}. 
\end{equation*}
This concludes the proof.
\end{proof}

We denote the closed embedding and open embedding by
\begin{align*}
    i\colon  &V=\bigcup_{\xi \in \bbC}\{\xi\}\times S_{\xi} \hookrightarrow Z\\j\colon  &U= {\bigcup_{z\in \bbC^*}\cH_z\times\{z\}} \hookrightarrow Z .
\end{align*}
Recall from \cref{remark: alternative description of unions of S1}, we have $V \simeq \bbC \times S^1$.
Then we have the following commutative diagram with $p, q$ natural projections to the first and second coordinates.
\begin{equation*}
\begin{tikzcd}
           & U \arrow[d, "j"] \arrow[rd, "q_{1}"] \arrow[ld, "p_{1}"'] &              \\
\bbC & Z \arrow[l, "p"'] \arrow[r, "q"]                                                                                                & \mathbb{C^*} \\
           & V \simeq \bbC\times S^1 \arrow[u, "i"'] \arrow[lu, "p_2"] \arrow[ru, "q_2"']                 &             
\end{tikzcd}
\end{equation*}
Let $\cF$ be a constructible sheaf on $\bbC$ with $R\Gamma(\cF)=0$, then \begin{equation*}
    (\Phi \cir \sigma\cir \Psi)(\cF) = Rp_{2!}q_2^!R^0q_{1*}p_{1}^*\cF. 
\end{equation*}
Let us unravel \cref{theorem: it sigma lt=id} using the alternative space $Z$, i.e.\ we want to show that 
\begin{equation*}
  Rp_{2!}q_2^!R^0q_{1*}p_1^*\cF = R^1p_{2*}q_2^*R^0q_{1*}p_1^*\cF =  R^1p_*i_*i^*q^*R^0q_*R^0j_{*}j^*p^*\cF\xrightarrow{\ \sim \ } \cF.
\end{equation*}
We will consider the following adjunction maps 
\begin{equation}\label{equation: 6.10}
R^1p_*i_*i^*q^*R^0q_*R^0j_{*}j^*p^*\cF\xrightarrow{\ \eta \ } R^1p_*i_*i^*R^0j_{*}j^*p^*\cF\xleftarrow{\ \theta \ } R^1(p\cir i)_*(p\cir i)^*\cF.
\end{equation} 
Our approach is structured as follows:
\begin{enumerate}[wide]
    \item We use the projection formula to show $ R^1(p\cir i)_*(p\cir i)^*\cF\xleftarrow{\ \sim \ }\cF$ in \cref{p2p2F=F}.
    \item We show $\theta$ in \eqref{equation: 6.10} is an isomorphism in \cref{lemma: theta is isomorphism}. Then we define \[R^1p_*i_*i^*q^*R^0q_*R^0j_{*}j^*p^*\cF\xrightarrow{\ \eta \ } R^1p_*i_*i^*R^0j_{*}j^*p^*\cF\xrightarrow{\ \theta^{-1} \ } R^1(p\cir i)_*(p\cir i)^*\cF.\]
    \item We show that the domain and target of $\eta$ have the same dimension in \cref{corollary: step(3) that dimension of psi are same}.
    \item We show $\eta $ is an injection, thus an isomorphism in \cref{proposition: step (4) psi is injectve}.
\end{enumerate}
Let us prove Step (1) that $\cF\xrightarrow{\ \sim \ }R^1(p\cir i)_*(p\cir i)^*\cF$.

\begin{lemma}\label{p2p2F=F}
Let $p_2$ denote the projection $\bbC\times S^1\to \bbC$ and let $\cF$ be a sheaf on $\bbC$. 
Then, for $i=0,1$, we have a natural isomorphism
\[\cF\xrightarrow{\ \sim \ } R^ip_{2*}p_2^*\cF.\]
\end{lemma}

\begin{proof}
For $i=0$, the result follows from the adjunction $\id \to R^0p_{2*}p_2^*$. For $i=1$, since $p_2$ is proper, by the projection formula (see \cite[Proposition 2.6.6]{KS_Sheaves_on_Manifolds}), we obtain an isomorphism
\begin{equation*}
   (Rp_{2*}\underline{k})\otimes^L \cF\xrightarrow{\ \sim \ }Rp_{2*}(\underline{k}\otimes^L p_2^*\cF).
\end{equation*}
Since they are sheaves of $k$-vector spaces, they are flat. So, the derived product $\otimes^L$ is equal to the usual tensor product.
Then we have
\begin{equation*}
    R^1p_{2*}p_2^*\cF = \cH^1(Rp_{2*} p_2^*\cF)=\cH^1\big(Rp_{2*}(\underline{k}\otimes_{\underline{k}} p_2^*\cF)\big).
\end{equation*} 
By the projection formula, we have 
\[ (Rp_{2*}\underline{k})\otimes_{\underline{k}} \cF  \xrightarrow{\ \sim \ } Rp_{2*}(\underline{k}\otimes_{\underline{k}} p_2^*\cF).\]
This implies that \[\cH^1\big( (Rp_{2*}\underline{k})\otimes_{\underline{k}} \cF\big)\xrightarrow{\ \sim \ } \cH^1\big(Rp_{2*}(\underline{k}\otimes_{\underline{k}} p_2^*\cF)\big).\]
Since 
\[\cH^1\big((Rp_{2*}\underline{k})\otimes_{\underline{k}} \cF\big) = R^1p_{2*}\underline{k}\otimes_{\underline{k}} \cF, \]
and $R^1p_{2*}\underline{k}= \underline{k}$, we have $   \cF\xrightarrow{\ \sim \ }   R^1p_{2*}p_2^*\cF.$
\end{proof}

Next we prove Step (2).
We start with the following \cref{lemma: pF to jjpF}.

\begin{lemma}\label{lemma: pF to jjpF}
Consider the following diagram, where $p_1,p$ are projections to the first coordinate, the canonical morphism $p^*\cF \rightarrow R^0j_*j^*p^*\cF$ is an isomorphism.
\begin{equation*}
\begin{tikzcd}
           & U= {\bigcup_{z\in \bbC^*}\cH_{z}\times \{z\}} \arrow[d, "j"]  \arrow[ld, "p_{1}"'] &              \\
\bbC & Z \arrow[l, "p"']                                                                                              \end{tikzcd}
\end{equation*}
\end{lemma}

\begin{proof}
To show that the morphism is injective, we consider the following short exact sequence of sheaves on $Z$, associated to the open embedding $j\colon  U\hookrightarrow Z$ and the closed embedding $i\colon  V\hookrightarrow Z$
\begin{equation*}
   0\longrightarrow \cH^0_V(p^*\cF)\longrightarrow p^*\cF\longrightarrow R^0j_*j^*p^*\cF,
\end{equation*}    
where $\cH^0_V(p^*\cF)$ is the sheafication of the presheaf sending an open set $W\subset Z$ to sections $s\in p^*\cF(W)$ supported on the boundary $V\subset Z$ . Note that the restriction of $p^*\cF$ to a slice $\xi\times \oD_{\xi}^*$ is the constant sheaf $\underline{\cF_\xi}$. If a section $s$ has support at a boundary point $(\xi,z)\in Z$, it also has support on the interior $\xi \times D_{\xi}^1$. Therefore, $\cH^0_V(p^*\cF)$ is zero, thus $p^*\cF\hookrightarrow R^0j_*j^*p^*\cF$ is injective.

Next, we show the domain and target of $p^*\cF\rightarrow R^0j_*j^*p^*\cF$ have the same stalkwise dimensions, which implies that it is an isomorphism by injectivity. This is always true on $U$ and we only need to consider the stalks on $V$. For any $(\xi, z)\in V$, we have \begin{equation*}
    (R^0j_*j^*p^*\cF)_{(\xi,z)} = \lim_{(\xi,z) \in U \text{ open in } Z } \cF\big( (p\cir j\cir j^{-1})(U)\big).
 \end{equation*}
For any open neighborhood $W\subset \bbC$ of $\xi$, we denote by $U = p^{-1}(W)$ the open neighborhood of $(\xi,z)$ in $Z$. Then $(p\cir j\cir j^{-1})(U)\subset W$ and we have an isomorphism 
\begin{equation*}
   \lim_{(\xi,z) \in U \text{ open in } Z } \cF\big((p\cir j\cir j^{-1})(U)\big) \xrightarrow{\ \sim \ } \cF_\xi.
\end{equation*} 
On the other hand, we have $p^*\cF_{(\xi,z)} = \cF_\xi$, thus $p^*\cF\rightarrow R^0j_*j^*p^*\cF$ is an isomorphism.
\end{proof}

\Cref{lemma: pF to jjpF} implies immediately the following.

\begin{lemma}\label{lemma: theta is isomorphism}
The second map in \eqref{equation: 6.10}  \[R^1p_*i_*i^*R^0j_{*}j^*p^*\cF\xleftarrow{\ \theta \ } R^1(p\cir i)_*(p\cir i)^*\cF\]
is an isomorphism.
\end{lemma}

Now we move on to Step (3). 
\begin{proposition}\label{h1 of Lz equals fz }
   Let $\cF\in \Constrz$ with singularities $\{c_1,\dots,c_n\}$. Denote by $\Psi(\cF) =(\cL, \cL_{<})$ its topological Laplace transform. Let $\xi$ be a point on $\bbC$. Then, we have
    \begin{equation*}
\dim H^1(S^1, \cL_{<\xi})= \dim\cF_\xi.
    \end{equation*}
\end{proposition} 

\begin{proof}
Let $\theta$ be a direction where none of $S_{c_i}$, with $c_i\neq \xi$, intersects $S_\xi$. Then by \cref{lemma: co stokes structure and decomposition H1}, \[        \dim H^1(S^1, \cL_{<\xi}) = \dim \cL_{<\xi, \theta}+\dim \cL_{<\xi, \theta+\pi}.
\]
By \cref{remark: stalk L at xi and theta}, we have
\[\cL_{<\xi, \theta}=\cF(\ocH_{\theta,\re(\xi e^{-i\theta})})\quad \text{and}\quad \cL_{<\xi, \theta+\pi}=\cF(\ocH_{\theta+\pi,\re(\xi e^{-i(\theta+\pi)})}).\]
On the boundaries of $\ocH_{\theta,\re(\xi e^{-i\theta})}$ and $\ocH_{\theta+\pi,\re(\xi e^{-i(\theta+\pi)})}$, there are no singular points, with the possible exception of $\xi$. Similar to the proof of \cref{Lemma: F at a point = F at two closed half-planes}, we show that 
\[\dim \cF(\ocH_{\theta,\re(\xi e^{-i\theta})})+\dim\cF(\ocH_{\theta+\pi,\re(\xi e^{-i(\theta+\pi)})})=\dim \cF_\xi,\]
completing the proof.
\end{proof}

\begin{corollary}\label{corollary: step(3) that dimension of psi are same} We have the following equality  
\begin{equation*}
\dim(R^1p_{2*}q_2^*R^0q_{1*}p_{1}^*\cF)_\xi  =\dim \cF_\xi.
\end{equation*}
\end{corollary}
\begin{proof}
    Since $p_2$ is proper, by the proper base change theorem \begin{equation*} (R^1p_{2*}q_2^*R^0q_{1*}p_{1}^*\cF)_\xi  = H^1(\{\xi \}\times S_{\xi},R^0q_{1*}p_1^*\cF|_{\{\xi \}\times S_{\xi}}) = H^1(S^1, \cL_{<\xi}).
    \end{equation*}
    By \cref{h1 of Lz equals fz }, we have \begin{equation*}
        \operatorname{dim
        }H^1(S^1, \cL_{<\xi}) = \dim\cF_\xi.
    \end{equation*}
    This concludes the proof.
\end{proof}

Finally, we move on to Step (4).

\begin{lemma}\label{fiber is constant sheaf}
The restriction of $i_*i^*R^0j_{*}j^*p^*\cF$ to $\{\xi\}\times S_\xi$ is the constant sheaf with fiber $\cF_\xi$.
\end{lemma}

\begin{proof}
By \cref{lemma: pF to jjpF}, we have $i^*p^*\cF\xrightarrow{\ \sim \ }i^*R^0j_{*}j^*p^*\cF$.
Since $\{\xi \}\times S_\xi \subset V$, it is just the pullback of $\cF$ to $\{\xi\}\times S_\xi$  by the projection to $\{\xi\}$, which is the constant sheaf $\underline{\cF_\xi}$.
\end{proof}

\begin{proposition}\label{proposition: step (4) psi is injectve} 
The map $\eta$ is injective.
\end{proposition}

\begin{proof}
Since $p_2$ is proper, by the proper base change theorem, we have
\begin{equation*} (R^1p_{2*}q_2^*R^0q_{1*}p_{1}^*\cF)_\xi  = H^1\big(\{\xi \}\times S_{\xi},R^0q_{1*}p_1^*\cF|_{\{\xi \}\times S_{\xi}}\big) = H^1\big(\{\xi\}\times S_\xi, \cL_{<\xi}\big).
\end{equation*}
By \cref{fiber is constant sheaf}, we have
\begin{equation*}
(R^1p_{2*}i^*R^0j_{*}j^*p^*\cF)_\xi = H^1\big(\{\xi\}\times S_\xi^1, i^*R^0j_*j^*p^*\cF|_{\{\xi\} \times S_\xi}\big) = H^1\big(\{\xi\}\times S_\xi, \underline{\cF_\xi}\big).
\end{equation*}

Next we show that $\eta_\xi$ is induced by a sheaf morphism $\cL_{<\xi} \longrightarrow \underline{\cF_\xi}$. Since taking stalk is an exact functor and by \cref{lemma: pF to jjpF}, the morphism
\[(R^1p_{2*}i^*q^*R^0q_*R^0j_{*}j^*p^*\cF)_\xi \longrightarrow (R^1p_{2*}i^*R^0j_{*}j^*p^*\cF)_\xi\]
is induced by taking $\cH^1$ of
\begin{equation*}
    k^*Rp_{2*}i^*q^*R^0q_*p^*\cF\longrightarrow k^*Rp_{2*}i^*p^*\cF,
\end{equation*}
where $k\colon  \xi\hookrightarrow \bbC$. We consider the following commutative diagram.

\begin{equation*}
\begin{tikzcd}
\{\xi\}\times S_\xi \arrow[d, "\tilde{p}_2"] \arrow[r, "\tilde{k}"] &  \bigcup_{\xi \in \bbC} \big(\{\xi \}\times S_{\xi}\big) \arrow[d, "p_2"] \\
\xi \arrow[r, "k"]                                                       & \bbC                                                                      
\end{tikzcd}
\end{equation*}
Since $p_2 $ is proper, by the proper base change theorem,
\begin{equation*}
k^*Rp_{2*}i^*q^*R^0q_*p^*\cF  =R\tilde{p}_{2*}\tilde{k}^*i^*q^*R^0q_*p^*\cF,
\end{equation*}
and similarly we have \begin{equation*}
    k^*Rp_{2*}i^*p^*\cF = R\widetilde{p}_{2*}\tilde{k}^*i^*p^*\cF.
\end{equation*}
By \cref{lemma: pF to jjpF},\[q^*R^0q_*p^*\cF|_{\{\xi\}\times S_\xi}\xrightarrow{\ \sim \ }R^0q_{1*}p_1^*\cF|_{\{\xi\}\times S_\xi}=\cL_{<\xi}.\] The canonical morphism  $\tilde{k}^*i^*q^*R^0q_{*}p^*\cF \longrightarrow \tilde{k}^*i^*p^*\cF
$ is the injection of sheaves
\begin{equation*}
q^*R^0q_*p^*\cF|_{\{\xi\}\times S_\xi} = \cL_{<\xi}\longrightarrow p^*\cF|_{\{\xi\}\times S_\xi} = \underline{\cF_\xi}.
\end{equation*}
Applying $R\tilde{p}_{2*}$, this induces a map on sheaf cohomology
\begin{equation*}
        f\colon H^1\big(\{\xi\} \times S_\xi, \cL_{<\xi}\big)\longrightarrow H^1\big(\{\xi\}\times S_\xi, \underline{\cF_\xi}\big).
\end{equation*}

Finally we show the injectivity of this map.
Let $\theta$ be a non-Stokes direction. Choose $\epsilon$ small enough, such that at any $\theta'\in U_1 \coloneqq (\theta-\epsilon,\theta+\epsilon)$, none of $S_{c_i}$ with $c_i\neq \xi$ intersects $S_\xi$. Let $U_2\coloneqq(\theta,\theta+\pi)$, $U_3\coloneqq(\theta+\pi-\epsilon,\theta+\pi+\epsilon)$ and $U_4\coloneqq(\theta+\pi,\theta+2\pi)$. Then $\cU$ is an open cover of $S^1$ consisting of small open intervals. By \cref{lemma: small open cover are Leray cover}, Čech cohomology groups are isomorphic to sheaf cohomology groups. By studying the induced maps on Čech cohomology groups, we show that $f$ is injective.
By the computation of dimensions, this map is also an isomorphism.
\end{proof}

\subsection{Natural isomorphism $ \Psi \cir (\Phi\cir \sigma)  \xrightarrow{\ \sim \ } \id$ }
\begin{theorem}\label{theorem: lt it sigma = id}
Given a co-Stokes structure of exponential type $(\cL, \cL_{<})$, there exists a natural isomorphism
\begin{equation*}
(\Psi \cir \Phi\cir \sigma) (\cL, \cL_{<}) \xrightarrow{\ \sim \ } (\cL, \cL_{<}).
\end{equation*}
\end{theorem}

For the proof, we will construct a map of sheaves on $\bbC^*$
\[\cG'\coloneqq (\sigma\cir \Psi \cir \Phi\cir \sigma) (\cL, \cL_{<}) \longrightarrow \cG\coloneqq\sigma (\cL, \cL_{<}),    
\]
and show that the induced map at infinity preserves the filtrations $\cL_<$ (see \cref{lemma: map at infinity preserves filtrations}). Consider the following commutative diagram,
\begin{equation*}
\begin{tikzcd}
\bbC & Z \arrow[l, "p"'] \arrow[r, "q"]                        & \bbC^* \\
           & V \arrow[u, "i"] \arrow[ru, "q_2"] \arrow[lu, "p_2" ']        &             \\ &  \bigcup_{z\in \bbC^*}\{\xi_z\}\times \{z\} \arrow[u,"k"] \arrow[ruu,"q_3"'] \arrow[luu,"p_3"] &
\end{tikzcd}
\end{equation*}
where $\xi_z$ denotes the center of the closed half-plane $\ocH_z$ (see \cref{notation: halfplanes}), and $q_3$ denotes the projection to $z$. Note that $q_3$ is a homeomorphism to $\bbC^*$.

Recall from \cref{Theorem: inverse Laplace transform is in constr} that $p_{2!}q_2^!\cG$ is a constructible sheaf $\cH^0(Rp_{2!}q_2^!\cG)$ placed in degree $0$. We now outline the construction of the map
\[
\cG'= R^0q_{*}p^* \cH^0(Rp_{2!}q_{2}^!\cG) \longrightarrow \cG.
\]
Consider the following adjunctions 
\begin{equation}\label{eq: adjunctions}
\begin{tikzcd}
R^0q_{*}p^*\cH^0(Rp_{2!}q_{2}^!\cG) \arrow[r, "i_*i^*"] & R^0q_{2*}p_2^*\cH^0(Rp_{2!}q_2^!\cG) \arrow[d,"k_*k^*" ]  & & \\ & 
R^0q_{3*}p_3^*\cH^0(Rp_{2!}q_2^!\cG)  & \arrow[l, "k_!k^!"'] R^0q_{3*}p_3^*\cH^0(Rp_{3*}q_3^*\cG)    \arrow[r] & \cG,
\end{tikzcd}
\end{equation}
where $q_3$ is a homeomorphism and the last map is given by \[R^0q_{3*}p_3^*\cH^0(Rp_{3*}q_3^*\cG) = R^0q_{3*}p_3^*R^0p_{3*}q_3^*\cG\rightarrow R^0q_{3*}q_3^*\cG = \cG.\]

We will show in \cref{kk1 is injective} that the adjunction $\id \rightarrow k_*k^*$ induces an injective map on $\cH^0$, and in \cref{kk2 is injective} that the adjunction $\id \leftarrow k_!k^!$ induces an injective map on $\cH^0$.
In \cref{lemma: lifts to psi}, we use the universal property of kernel to show that we have a lift $\psi$ 
\begin{equation*}
\begin{tikzcd}
& R^0q_{2*}p_2^*\cH^0(Rp_{2!}q_2^!\cG) \arrow[d,"k_*k^*" ]  \arrow[dr, "\psi", dashed] & \\ & R^0q_{3*}p_3^*\cH^0(Rp_{2!}q_2^!\cG)  & \arrow[l, "k_!k^!"'] R^0q_{3*}p_3^*\cH^0(Rp_{3*}q_3^*\cG).   
\end{tikzcd}
\end{equation*}
We will show in \cref{lemma: at ininfity is isomorphism} that the map on the local system at infinity is an isomorphism, and in \cref{lemma: map at infinity preserves filtrations} that this map preserves the filtrations $\cL_<$.

\begin{lemma}\label{kk1 is injective}
The map
\[r_1\colon R^0q_{2*}p_2^*\cH^0(Rp_{2!}q_2^!\cG)\longrightarrow R^0q_{3*}p_3^*\cH^0(Rp_{2!}q_2^!\cG),\]
induced by the adjunction $\id \rightarrow k_*k^*$, is injective on $\cH^0$.
\end{lemma}

\begin{proof}
Recall from \cref{Theorem: inverse Laplace transform is in constr} that $p_{2!}q_2^!\cG = \cF\in \Constrz$. 
By \cref{remark: alternative description of unions of S1}(1), $q_2^{-1}(z)=\{l_z\}\times \{z\}$ (see \cref{notation: halfplanes}). 
Then the induced map on the stalks at $z$ is the restriction map
\[r_{1,z}\colon \cF(\ocH_z)\longrightarrow \cF(l_z),\]
which is injective.
\end{proof}

Consider the following diagram. 
\begin{equation*}
\begin{tikzcd}
           & V\setminus\bigcup_{z\in\bbC^*}\{\xi_z\}\times\{z\} \arrow[ld, "p_4"'] \arrow[rd, "q_4"] \arrow[d, "j"'] &              \\
\bbC & V \arrow[l, "p_2"'] \arrow[r, "q_2"]                                                               & \bbC^* \\
           & \bigcup_{z\in\bbC^*}\{\xi_z\}\times\{z\} \arrow[u, "k"] \arrow[ru, "q_3"'] \arrow[lu, "p_3"]   &     
\end{tikzcd}
\end{equation*}
Since $p_3$ is proper and $q_3$ is a homeomorphism, we have
\begin{equation*}
    Rp_{3*}q_3^*\cG = Rp_{3!}q_{3}^!\cG = Rp_{2!}Rk_!k^!q_2^!\cG.
\end{equation*}
The exact sequence of cohomology sheaves associated to the fiber sequence
\begin{equation}\label{eq: fiber sq 1}
    Rp_{3*}q_3^*\cG=Rp_{2!}Rk_!k^!q_2^!\cG\xrightarrow{\ r_2 \ }     Rp_{2!}q_2^!\cG\longrightarrow Rp_{2!}Rj_*j^*q_2^!\cG
\end{equation}
gives rise to a map $\cH^0(Rp_{3*}q_3^*\cG)\longrightarrow \cH^0(Rp_{2!}q_2^!\cG).$
\begin{lemma}\label{kk2 is injective}
The map \begin{equation}\label{eq: adjunction 1}
    R^0q_{3*}p_3^* \cH^0(Rp_{3*}q_3^*\cG)\longrightarrow R^0q_{3*}p_3^*\cH^0(Rp_{2!}q_2^!\cG),
\end{equation}
induced by adjunction $k_!k^!\rightarrow \id$, is injective.
\end{lemma}
\begin{proof}
By \cref{eq: fiber sq 1}, it suffices to prove that \[\cH^{-1}(Rp_{2!}Rj_*j^*q_2^!\cG)=\cH^{0}(Rp_{2*}Rj_*j^*q_2^*\cG)=\cH^{0}(Rp_{4*}q_4^*\cG)=R^0p_{4*}q_4^*\cG=0.\] 
For $ \xi = re^{i\theta}\neq 0$, we have $ q_4\cir p_4^{-1}(\xi) = S_\xi \setminus \{ e^re^{i\theta}, e^{-r}e^{i(\theta+\pi)}\}.$ For $\xi=0$, we have $p_3^{-1}(0) = \{0\}\times S_{\xi=0 }$. 
For a directed system of open neighborhoods $U_i$ of $\xi$, we have $(R^0p_{4*}q_4^*\cG)_\xi = \lim_{\xi\in U} (R^0p_{4*}q_4^*\cG)(U)=0$. Therefore, $r_2$ is injective on $\bbC$, completing the proof.
\end{proof}

\begin{lemma}\label{lemma: lifts to psi}
Consider the following diagram, 
\begin{equation*}
\begin{tikzcd}
&R^0q_{2*}p_2^*\cH^0(Rp_{2!}q_2^!\cG) \arrow[d,"r_1" ]  \arrow[dr, "\psi", dashed] & \\ R^0q_{3*}p_3^*\cH^0(Rp_{2!}Rj_*j^*q_2^!\cG)  & R^0q_{3*}p_3^*\cH^0(Rp_{2!}q_2^!\cG)  \arrow[l, "r_3"'] & \arrow[l,  "r_2"'] R^0q_{3*}p_3^*\cH^0(Rp_{3*}q_3^*\cG) \longleftarrow 0, 
\end{tikzcd}
\end{equation*}
where $r_2$ is induced by the adjunction $k_!k^!\rightarrow \id$ on $\cH^0$. Then the bottom sequence is exact, the composition $r_3\circ r_1 =0$ and $r_1$ lifts to $\psi$.
\end{lemma}
\begin{proof}
By \cref{kk2 is injective}, the fiber sequence 
\[ Rp_{3*}q_3^*\cG\longrightarrow Rp_{2!}q_2^!\cG\longrightarrow Rp_{2!}Rj_*j^*q_2^!\cG \]
induces an exact sequence of cohomology sheaves
\[0\rightarrow \cH^0(Rp_{3*}q_3^*\cG)\longrightarrow \cH^0(Rp_{2!}q_2^!)\cG\longrightarrow \cH^0(Rp_{2!}Rj_*j^*q_2^!\cG).\] 
Since $R^0q_{3*}$ is left exact, the bottom sequence is exact.
By naturality, $r_3\circ r_1$ factors as follows
\begin{equation*}
\begin{tikzcd}
R^0q_{2*}p_2^*\cH^0(Rp_{2!}Rj_*j^*q_2^!\cG) \arrow[d]& R^0q_{2*}p_2^*\cH^0( Rp_{2!}q_2^!\cG) \arrow[d,"r_1" ] \arrow[l] & \\ R^0q_{3*}p_3^*\cH^0(Rp_{2!}Rj_*j^*q_2^!\cG) & R^0q_{3*}p_3^*\cH^0( Rp_{2!}q_2^!\cG)  \arrow[l, "r_3"']. 
\end{tikzcd}
\end{equation*}
We show that $\cH^0(Rp_{2!}Rj_*j^*q_2^!\cG)=0$. Using $Rp_{2!}=Rp_{2*}$ and $q_2^!=q_2^*[1]$, we have \[\cH^0(Rp_{2!}Rj_*j^*q_2^!\cG)=\cH^1(Rp_{4*}q_4^*\cG) = R^1p_{4*}q_4^*\cG.\] 
Thus, for any $z\in\bbC^*$, we have
\[\big(R^0q_{2*}p_2^*\cH^0(Rp_{2!}Rj_*j^*q_2^!\cG)\big)_z = \big(R^0q_{2*}p_2^*(R^1p_{4*}q_4^*)\cG\big)_z=(R^1p_{4*}q_4^*\cG)(l_z).\] 
For each $\xi\in l_z$, we take a small open neighborhood $U$ of $\xi$ such that $(R^1p_{4*}q_4^*\cG)_\xi=H^1\big(p_4^{-1}(U),q_4^*\cG\big)$. 
Notice that $p_4^{-1}(U)$ is a disjoint union of two open sets $U_1$ and $U_2$, where each $U_i$ is connected, and $p_4|_{U_i}: U_i \rightarrow U$ is an open map with connected fibers. 
So for any sheaf $\cE$ on $U$ and open set $V\subset U$, we have $(p_4|_{U_i})^*\cE(V) = \cE(p_4(V))$. Then by taking a flasque resolution of $\cG$, we see that $H^1(U_1,q_4^*\cG) = H^1(q_4(U_1),\cG)$. 
Since $\cG$ is isomorphic to a direct sum of constants sheaves and extensions by zeros of constant sheaves on $q_4(U_1)$, we have $H^1(q_4(U_1),\cG)=0$. 

Then $r_3\circ r_1$ factors through $R^0q_{2*}p_2^*\cH^0(Rp_{2!}Rj_*j^*q_2^!\cG)=R^0q_{2*}p_2^* (R^1p_{4*}q_4^*)\cG =0$, and hence it is zero.
Since $R^0q_{3*}p_3^*\cH^0(Rp_{3*}q_3^*\cG)$ is the kernel of $r_3$, we conclude that $r_1$ lifts to a map $\psi$.
\end{proof}
By composing $\psi$ with the adjunction maps in \eqref{eq: adjunctions}, we obtain a map
\[\Lambda\colon \cG'=R^0q_{*}p^*\cH^0(Rp_{2!}q_{2}^!\cG)   \xrightarrow{\psi_1} R^0q_{2*}p_2^*\cH^0(Rp_{2!}q_2^!\cG) \xrightarrow{\psi} R^0q_{3*}p_3^*\cH^0(Rp_{3*}q_3^*\cG) \xrightarrow{\psi_2} \cG.\]
Next, we show that the restriction of $\Lambda$ to a circle near infinity is an isomorphism. 
Let $i_\infty\colon S^1\rightarrow \bbC^*$ denote the closed embedding of a circle encircling all $S_{c_i}$ for $1\leq i\leq n$.

\begin{lemma}\label{lemma: at ininfity is isomorphism}
The restriction of $\Lambda$ to infinity 
\[i_\infty^*\Lambda\colon i_\infty^*\cG'\rightarrow i_\infty^*\cG\] is an isomorphism. 
\end{lemma}
\begin{proof}
Let $a = e^\lambda e^{i\theta}\in i_\infty(S^1)$ and $b = e^{-\lambda} e^{i(\theta+\pi)}\in i_\infty(S^1)$.
Recall from \cref{Theorem: inverse Laplace transform is in constr} that $p_{2!}q_2^!\cG$ is a constructible sheaf in $\Constrz$, which we denote by $\cF$. The restriction of $\psi_1$ at $a$ is
\[\psi_{1,a}\colon \cF(\ocH_a)\longrightarrow \cF(l_a),\]
which is an isomorphism for all $a\in i_\infty(S^1)$.
By \cref{lemma: sheaf cohomology of Lz and points},
\[\dim \cF(l_a) = \dim \cG_a\oplus \dim \cG_b = \dim \cG_a ,\]
where the last equality holds since $a\in i_\infty(S^1)$ implies that $\cG_b=0$.

The restriction of $\psi_2$ at $a$ is the projection
\[\psi_{2,a}\colon \cG_{a}\oplus \cG_b\longrightarrow \cG_a,\]
which is also an isomorphism since $\cG_b=0$.

By \cref{kk1 is injective}, $r_1$ is injective. Since $r_1= r_2\circ \psi$, $\psi$ is also injective. So the composition $ \psi_{1,a}\circ \psi_a\circ \psi_{2,a}: \cF(\ocH_a)\longrightarrow \cG_a$ is an injection, and thus an isomorphism. We conclude the proof.
\end{proof}

\begin{lemma}\label{lemma: map at infinity preserves filtrations}
The map $i_\infty^*\Lambda$ of local systems preserves the filtrations at any $\theta\in S^1$, i.e.\ for any $\xi \in \bbC$, we have $i_\infty^*\Lambda(\cL_{ <\xi,\theta}')\subset \cL_{<\xi,\theta}$.
\end{lemma}

\begin{proof}
For any $a=e^\lambda e^{i\theta}\in S_\xi \subset \bbC^*$ with argument $\theta$, we have $L_{<\xi,\theta}=\cG_a$. 
Let $R_a$ denote the ray from $a$ as in \cref{notation: Iz Rz}.
Let $s\colon [0,\infty)\rightarrow R_a$ be the parametrization $r\mapsto (r+e^\lambda )e^{i\theta}$ and $r_\infty>0$ such that $s(r_\infty)\in i_\infty(S^1)\cap R_a$.
Denote by $\pi\colon [0,\infty)\rightarrow \{0\}$ the projection.  

Consider the following commutative diagram.
\[\begin{tikzcd}
s^*\cG' \arrow[r]                       & s^*\cG                       \\
\pi^* R^0\pi_* s^*\cG' \arrow[r] \arrow[u] & \pi^* R^0\pi_* s^*\cG \arrow[u]
\end{tikzcd} \]
The stalk at $r_\infty\in [0,\infty)$ of the above diagram is 
\[\begin{tikzcd}
\cL'_{\theta}\arrow[r]                         & \cL_{\theta}                        \\
{\cG_{a}'}=\cL_{<\xi,\theta}' \arrow[u,hook] \arrow[r] &{\cG_{a}}=\cL_{<\xi,\theta} \arrow[u, hook],
\end{tikzcd}\]
where $\cG(R_a)=\cG_a$ and $\cG'(R_a)=\cG_a'$.
This implies that $i_\infty^*\Lambda$ preserves the filtrations.
\end{proof}

This completes the proof of \cref{theorem: lt it sigma = id}.
We end this section with an interesting corollary.

\begin{lemma}\label{lemma: local system at infinity vanish implies trivial constructible sheaf}
Let $\cF\in\Constrz$ be a constructible sheaf on $\bbC$ with $R\Gamma(\cF)=0$. If the monodromy at infinity $T_\infty$ of $\cF$ is the identity, its monodromy around each singularity $T_i$ is also the identity.
\end{lemma}

\begin{proof}
In \cite[\S 2.3]{KKP_Hodge_theoretic_aspect_of_mirror_symmetry}, the authors give an equivalent description of $\Constrz$ using monodromies. By choosing small disks $D_1, \dots,D_n$ around singularities $c_1, \dots,c_n$, and paths $\gamma_1, \dots, \gamma_n$ from a nonsingular point $c_0$ to a boundary point on $D_i$, one can identify the stalks $\cF_{c_i}$ with subspaces of $\cF_{c_0}$. It follows from the condition $R\Gamma(\cF)=0$ that the map \begin{equation*}
    \cF_{c_0}\xrightarrow{\ \sim \ } \bigoplus \cF_{c_0}/\cF_{c_i}
\end{equation*} 
induced by projections is an isomorphism. 
The monodromy of $\cF$ around the singular point $c_i$ can be written in a simple form under the above identification. The lemma then follows from induction on the number of singularities (see \cite[Corollary 7.1.2]{Sabbah_vanishing_cycles_of_polynomial_maps_topology_hodge_structure_d_modules}).
\end{proof}

\begin{corollary}\label{cor: trivial Stokes structure}
Let $(\cL, \cL_{\leq})$ be a Stokes structure  on $S^1$. If $\cL$ is trivial, then $(\cL, \cL_{\leq})$ is a trivial Stokes structure.
\end{corollary}
\begin{proof}
By the equivalence of categories in \cref{theorem: lt it sigma = id}, it suffices to prove that if $\cL$ is trivial, the monodromies of $\cF=\Phi(\cL, \cL_<)$ along the singularities $c_i$'s are trivial, which follows from \cref{lemma: local system at infinity vanish implies trivial constructible sheaf}.
\end{proof}

\section{B-model nc-Hodge structures from mirror symmetry}\label{section: B-model}

In this section, we give the example of B-model nc-Hodge structures in two ways related by Fourier transform on the de Rham data and topological Laplace transform on the Betti data.
In general, Fourier-Laplace transforms lead to a dual description of nc-Hodge structures, which we formulate in \cref{sec: dual description}.

\subsection{The dual description of nc-Hodge structures} \label{sec: dual description}

We refer to \cite[Definitions 2.5, 2.14]{KKP_Hodge_theoretic_aspect_of_mirror_symmetry}, for the definition of the nc-Hodge structure.
The dual description in \cref{definition: dual notion of nc-Hodge} is modeled on \cite[Theorems 2.33, 2.35]{KKP_Hodge_theoretic_aspect_of_mirror_symmetry}.
We will prove an equivalence of categories between the original and the dual descriptions in \cref{theorem: dual description of nc-Hodge is an equivalence}, as an application of the equivalence in \cref{theorem: const0 and Stokes are equivalent 6.1}.

Let us first recall the vanishing cycle decomposition of a constructible sheaf.
\begin{definition}
Let $\cF\in\Constrz$ with singularities $\{c_1,\dots,c_n\}$.
For any $z\in\bbC$, let $i\colon \{z\}\hookrightarrow\mathbb{C}$ denote the inclusion, $j\colon D^*\hookrightarrow \mathbb{C}$ the inclusion of a small punctured disk centered at $z$, not containing any singularity, and $e\colon U \to D^*$ the universal covering.
We define
\begin{align*}
         \psi_z(\cF)&\coloneqq i^* R^0j_*R^0e_*e^*j^*\cF,\\ \phi_z(\cF)&\coloneqq \mathrm{cone}(i^*\cF \longrightarrow i^*R^0j_*R^0e_*e^*j^*\cF),
\end{align*}
called the \emph{complex of nearby (resp.\ vanishing) cycles} of $\cF$ at $z$.
\end{definition}

\begin{definition}\label{def: vanishing cycle decomposition}
Let $\cF\in\Constrz$ with singularities $\{c_1,\dots,c_n\}$. 
Choosing a non-singular point $c_0\in\bbC$ and Gabrielov paths as in \cite[\S 2.3.2]{KKP_Hodge_theoretic_aspect_of_mirror_symmetry}, we have natural inclusions $\cF_{c_i} \subset \cF_{c_0}$ and an isomorphism
\[\cF_{c_0}\xrightarrow{\ \sim \ } \bigoplus_{i=1}^n \cF_{c_0}/\cF_{c_i}, \]
which we call the \emph{vanishing cycle decomposition} of $\cF_{c_0}$.
Let $U_i\coloneqq \cF_{c_0}/\cF_{c_i}$.
It is canonically identified with $\phi_{c_i}(\cF)$.
We denote by $T_{ii}\colon \colon U_i\longrightarrow U_i$ the monodromy action around $c_i$.
\end{definition}

The following lemma relates the vanishing cycle decomposition with the graded local system associated to the Stokes structure.

\begin{lemma}\label{lemma: compute Stokes matrices from Tij}
Let $\cF\in\Constrz$  with singularities $\{c_1,\dots,c_n\}$, $(\cL, \cL_<)\coloneqq\Psi(\cF)$ its topological Laplace transform, and $(\cL,\cL_\leq)$ the Stokes structure equivalent to the co-Stokes structure $(\cL,\cL_<)$.
The associated graded local system $\gr_{c_i}\cL$ of $(\cL, \cL_{\leq})$ has the following monodromy representation
\begin{equation*}
    \gr_{c_i}\cL = (U_i,T_{ii}).
\end{equation*}
\end{lemma}

\begin{proof}
Let $\theta$ be a non-Stokes direction.
Up to relabeling, we may assume that
\[ c_1<_\theta\dots<_\theta c_n.\]
Let $\epsilon$ be small enough such that $(\theta-\epsilon,\theta+\epsilon)$ does not contain any Stokes directions. 
Let $I_1\coloneqq (\theta-\epsilon, \theta+\pi+\epsilon)$ and $I_2\coloneqq (\theta+\pi-\theta,\theta+2\pi+\epsilon)$.
They are two good open intervals.
By \cref{proposition: Stokes filtration trivializes over good open intervals}, we have a unique decomposition
\[\bigoplus_{i=1}^n(\gr_{c_i}\cL)(I_1) \xrightarrow{\ \sim \ }\cL(I).\]
Under the decomposition $\cL_\theta=\bigoplus_{i=1}^n(\gr_{c_i}\cL)_\theta$, the monodromy $T\colon \cL_\theta\xrightarrow{ \ \sim \ }\cL_\theta$ of $\cL$ can be written in the following form

\begin{equation*}
T = \begin{pmatrix}
1 & S_{12} & S_{13} & \dots & S_{1n} \\
0 & 1 & S_{23} & \dots & S_{2n} \\
0 & 0 & \ddots & \ddots & \vdots \\
\vdots & \vdots & \ddots & 1 & S_{n-1,n} \\
0 & 0 & \dots & 0 & 1
\end{pmatrix}
\cdot
\begin{pmatrix}
Q_{11} & 0 & 0 & \dots & 0 \\
Q_{21} & Q_{22} & 0 & \dots & 0 \\
Q_{31} & Q_{32} & \ddots & \ddots & \vdots \\
\vdots & \vdots & \ddots & \ddots & 0 \\
Q_{n1} & Q_{n2} & \dots & \dots & Q_{nn}
\end{pmatrix},
\end{equation*}
where $S_{ij}\colon (\gr_{c_j}\cL)_\theta\rightarrow (\gr_{c_i}\cL)_\theta$ and $Q_{ij}\colon (\gr_{c_j}\cL)_\theta\rightarrow (\gr_{c_i}\cL)_\theta$ are the Stokes data as in \cite[Definition 2.6]{Hertling_Sabbah_examples_of_non-commutative_hodge_structures}.
In particular, the monodromy of $\gr_{c_i}\cL$ is \[Q_{ii}\colon (\gr_{c_i}\cL)_\theta\rightarrow (\gr_{c_i}\cL)_\theta.\] 

Let $c_0$ be a nonsingular point of $\cF$ near infinity.
Consider the vanishing cycle decomposition of $\cF_{c_0}$ as in \cref{def: vanishing cycle decomposition}
\[\cF_{c_0}\xrightarrow{\ \sim \ }\bigoplus_{i=1}^n \cF_{c_0}/\cF_{c_i}.\]
Then the monodromy $T_i$ around each singular point $c_i$ can be written as 
\[T_i = \begin{pmatrix}
1 & 0 & \dots  & T_{1i} & 0 & \dots & 0 \\
0 & 1 & \dots  & T_{2i} & 0 & \dots & 0 \\
\vdots & \vdots & \ddots  & \vdots & \vdots & \ddots & \vdots \\
0 & 0 & \dots  & T_{ii} & 0 & \dots & 0 \\
0 & 0 & \dots  & T_{i+1,i} & 1 & \dots & 0 \\
\vdots & \vdots & \ddots  & \vdots & \vdots & \ddots & \vdots \\
0 & 0 & \dots & T_{ni} & 0 & \dots & 1
\end{pmatrix},
\]
where $T_{ij}\colon \cF_{c_0}/\cF_{c_j}\longrightarrow \cF_{c_0}/\cF_{c_i}.$

Choose $r$ such that $\ocH_{r,\theta}$ contains $c_0$ but does not contain any singular point $c_i$.
Then we have the restriction map,
\[\rho\colon \cL_{\theta}= \cF(\ocH_{r,\theta}) \xrightarrow{\ \sim \ }\cF_{c_0}\]
which sends $(\gr_{c_i}\cL)_\theta$ to $\cF_{c_0}/\cF_{c_i}$. 
Thus, $\rho$ is compatible with the decomposition and $\rho$ is diagonal. 
By the construction of topological Laplace transform, we have 
\begin{equation*}
      T = \rho^{-1} \cir T_1 \cir \dots \cir T_n \cir \rho.
\end{equation*}
By induction, we conclude that \[Q_{nn}=\rho^{-1}\circ T_{nn}\cir\rho.\] 
\end{proof}

Let $V$ be a finite dimensional $\bbC$-vector space with an automorphism $T$.
By \cite[Lemma 2.36]{KKP_Hodge_theoretic_aspect_of_mirror_symmetry}, it is equivalent to a finite dimensional $\bbC\dbp{u}$-vector space $R$ with a meromorphic connection $\nabla$ having a regular singularity.
We denote this equivalence by $(R,\nabla)\coloneqq (\RH^\formal)^{-1}(V,T)$.

\begin{definition}\label{definition: dual notion of nc-Hodge}
A \emph{dual description of nc-Hodge structure of exponential type} is a tuple $\big(\{H_i, \iso_i\}_{i=1}^n,\allowbreak \cF\big)$ consisting of the following data:
\begin{enumerate}[wide]
   \item Each $H_i$ is a free $\bbC\dbb{u}$-module.
   \item $\cF$ is a $\bbQ$-constructible sheaf on $\bbC$ with $R\Gamma(\cF)=0$.
   \item \label{definition: extension} For each $i=1, \dots, n$, let $(\phi_{c_i}(\cF),T_{ii})$ be as in \cref{def: vanishing cycle decomposition}.
   It is equivalent to a pair $(R_i, \nabla_i)$, where $R_i$ is a $\bbC\dbp{u}$-module, and $\nabla_i$ is a regular meromorphic connection
 (see \cite[Lemma 2.36]{KKP_Hodge_theoretic_aspect_of_mirror_symmetry}).
   Then $\iso_i$ is an isomorphism between $R_i$ and $H_i\otimes_{\bbC\dbb{u}}\bbC \dbp{u}$, such that $H_i$ is a lattice of Poincaré rank $1$. 
\end{enumerate}
\end{definition}

\begin{theorem}\label{theorem: dual description of nc-Hodge is an equivalence}
We have an equivalence between the category of nc-Hodge structures of exponential type without the opposedness axiom and the category of dual descriptions of nc-Hodge structures of exponential type.

\end{theorem}
\begin{proof}
By \cref{theorem: const0 and Stokes are equivalent 6.1}, a co-Stokes structure of exponential type $(\cL,\cL_<)$ is equivalent to a constructible sheaf $\cF\in \Constrz$ via the topological Laplace transform functors.
Under the irregular Riemann-Hilbert correspondence, $(\cL,\cL_<)\otimes \bbC$ corresponds to an algebraic vector bundle with a connection $(H',\nabla)$ on $\bbA^1\setminus 0$ (see \cite[Theorems 4.2]{Malgrange1982_la_classification_des_connexions_irregulieres_a_une_variable}).

We have a formal decomposition at $0$, 
\[    (H', \nabla)\otimes \bbC\dbp{u}\simeq \bigoplus_{i=1}^n\big(e^{c_i/u}\otimes(R_i, \nabla_i)\big),
\]
where each $(R_i,\nabla_i)$ is an algebraic vector bundle with a connection on $\bbA^1\setminus 0$ having regular singularities at $0$ and $\infty$, and $e^{c_i / u}$ denotes the rank-one trivial bundle over $\bbA^1$ with the connection $d-d(c_i / u)$.
We have the following isomorphisms 
\[\RH^{-1}(\cL\otimes_\bbQ\bbC)\xrightarrow{\ \sim\ } (H',\nabla)^{\an} \quad\text{and}\quad \RH^{-1}(\gr_{c_i}\cL\otimes_\bbQ\bbC)\xrightarrow{\ \sim \ } (R_i,\nabla_i)^{\an},\]
where $\RH$ denote the Riemann-Hilbert correspondence (see \cite[Theorems 4.2, 4.4]{Malgrange1982_la_classification_des_connexions_irregulieres_a_une_variable}).
By \cref{lemma: compute Stokes matrices from Tij}, $\gr_{c_i}\cL$ has the monodromy representation as $(\phi_{c_i}(\cF),T_{ii})$. 
Then we have 
\[  \RH^{-1}\big(\phi_{c_i}(\cF),T_{ii}\big)\xrightarrow{\ \sim \ } (R_i,\nabla_i)^{\an}.\] 

We denote by $(R_i^\formal,\nabla_i^\formal)$ the formal completion of the germ at zero of $(R_i,\nabla_i)^\an$. Since taking formal completion is an equivalence for regular connections, we have an $\bbC\dbp{u}$-isomorphism  \[\iso_i\colon (\RH^\formal)^{-1}\big(\phi_{c_i}(\cF),T_{ii}\big)\xrightarrow{\ \sim \ } (R_i^\formal,\nabla_i^\formal) \]  
By \cite[Lemma 8.2]{Hertling_Sevenheck_nilpotent_orbits_of_a_generalization_of_hodge_structures} and \cite[\S 2.3.1]{KKP_Hodge_theoretic_aspect_of_mirror_symmetry}, choosing an extension $H$ of $H'$ to $\bbA^1$ such that $(H,\nabla)$ satisfies the nc-filtration axiom in \cite[Definition 2.5]{KKP_Hodge_theoretic_aspect_of_mirror_symmetry} is equivalent to choosing a $\bbC\dbb{u}$-module $H_i$ for each $R_i$ such that $(H_i,\nabla_i)$ satisfies the \cref{definition: dual notion of nc-Hodge}(\ref{definition: extension}).
Thus, we conclude the proof.
\end{proof}

\subsection{Example of B-model}\label{subsection: Example of B-model}

Let $Y$ be a smooth complex quasi-projective variety and $f\colon Y \rightarrow \bbA^1$ a proper map.
We will recall the B-model nc-Hodge structure associated to $(Y,f)$ in \cref{definition: B model nc-Hodge}.
We give an explicit description of the associated Stokes structure via relative cohomology in \cref{lemma: comparison of Betti data} using the topological Laplace transform, and then deduce the $\bbQ$-structure axiom in \cref{proposition: de Rham and Betti data compatible}.
In \cref{gluing is dual to GM}, we give the dual description of the B-model nc-Hodge structure under the equivalence in \cref{theorem: dual description of nc-Hodge is an equivalence}, and relate it to the construction via Gabrielov paths in \cite[\S 3.2]{KKP_Hodge_theoretic_aspect_of_mirror_symmetry}.
The key is \cref{lemma: comparison of Betti data}, which compares the Betti data via the topological Laplace transform.

\begin{definition}\label{definition: B model nc-Hodge}
The B-model nc-Hodge structure associated to $(Y,f)$ consists of $(H, \cE_B, \iso)$ defined as follows: 
\begin{enumerate}[wide]
    \item $H$ is a $\bbZ/2$-graded algebraic vector bundle over $\bbA^1$,
    \[ H\coloneqq \mathbb{H}^\bullet\big(Y, (\Omega_Y^{\bullet}[u], ud-df )\big),\]
     (see \cite[Theorem 1]{Sabbah_on_a_twisted_de_rham_complex_I}, \cite[Theorem 4.22]{Ogus_Vologodsky_non_abelian_hodge_theory_in_characteristic_p} and \cite[Lemma 3.6]{KKP_Hodge_theoretic_aspect_of_mirror_symmetry}).
    \item $\cE_B$ is the pullback to $\bbA^1\setminus 0$ of the following local system on $S^1$,
    \[ \cL_\theta =  H^\bullet\big(Y,f^{-1}(\infty_\theta);\bbQ\big), \]
    where $\infty_\theta$ denotes a point near infinity with argument $\theta$.
    \item
    $\iso$ is an isomorphism of holomorphic vector bundles over $\bbA^1\setminus 0$
    given in \cite[\S 7.e]{Sabbah_introduction_to_Stokes_structure},
    \[\iso\colon \cE_B\otimes\cO^\an_{\bbA^1\setminus 0} \xrightarrow{\ \sim\ } H^\an|_{\bbA^1\setminus 0}.\]
    .
\end{enumerate}
\end{definition}

\begin{remark}
    Let us comment on the axioms of nc-Hodge structure defined in \cite[Definition 2.14]{KKP_Hodge_theoretic_aspect_of_mirror_symmetry}.
    \begin{enumerate}[wide]
        \item Let $j\colon\bbA^1 \setminus 0\hookrightarrow \bbA^1$ denote the open immersion, $i\colon \bbA^1 \setminus 0\rightarrow \bbA^1 \setminus 0$ the map $u\mapsto 1/u$, and $\FT$ the Fourier transform of $D$-modules on $\bbA^1$.
        By \cite[\S 1.1]{Sabbah_on_a_twisted_de_rham_complex_I}, for any $q\in\bbN$, we have 
        \[\Big(\mathbb{H}^{q+\dim Y}\big(Y, (\Omega_Y^{\bullet}[u^{-1}], d-d f /u)\big), \nabla_{\partial_u}\coloneqq \partial_u+f/u^2\Big) \simeq i_* j^* \FT (\cH^q f_* \cO_Y ).
        \]
        By \cite[Proposition 2.2]{Saito_on_the_cohomology_of_a_general_fiber_of_a_polynomial}, $\cH^q f_* \cO_Y$ is a holonomic $D$-module with regular singularities.         Recall that the Laplace transform of a regular holonomic $D$-module has a regular singularity at $0$ and a singularity of exponential type at infinity (see \cite[\S IX-XI]{Malgrange91_equation_differentielle_a_coefficient_polynomiaux} and \cite[Theorem 2.10.16]{Katz_Exponential_sums_and_differential_equations}).
        This implies the nc-filtration axiom. 
        \item The $\bbQ$-structure axiom is proved in \cite[Theorem 6.3]{sabbah_non-commutative_hodge_structures} in the affine case.
        We extend the proof to the general case in \cref{proposition: de Rham and Betti data compatible}.
        \item The opposedness axioms is conjectural in general (see \cite{Reichelt_non_affine_LG_model}).
    \end{enumerate}
\end{remark}

Let us recall the projector $\Pi$ on D-modules (or perverse sheaves) from \cite[\S 4.2]{Kontsevich_Soibelman_cohomological_hall_algebra_exponential_hodge_strucutures} (see also \cite{Fresan_Exponential_Motives}).
Let $\star_+$ denote the additive convolution of $D$-modules (or perverse sheaves) on $\bbA^1$ as in \cite[\S 12]{Katz_Exponential_sums_and_differential_equations}.
For $D$-modules, we have an exact idempotent endofunctor on the category of holonomic $D$-modules on $\bbA^1$ with regular singularities, given by
\[
\Pi(M) \coloneqq M\star_+ j_!(\cO_{\bbA^1\setminus 0}[1]).
\]
Let $\Perv$ denote the category of perverse sheaves on $\bbA^1$, and $\Perv_0$ the subcategory consisting of perverse sheaves with vanishing cohomology.
We have an the exact idempotent endofunctor
\begin{align*}
    \Pi \colon \Perv &\longrightarrow \Pervzero \subset \Perv \\
    \cF &\longmapsto \cF\star_+ 0j_!(\bbC_{\bbA^1\setminus \{0\}}[1]).
\end{align*}
It then follows that $\Pi$ is compatible with Riemann-Hilbert correspondence and it satisfies $ \FT \cir  \Pi= (j_*j^*)\cir \FT$. One can also extend $\star_+$ and $\Pi$ to the bounded derived category of D-modules and perverse sheaves (see \cite[Definition 2.4.1]{Fresan_Exponential_Motives}).

\begin{lemma}\label{lemma: GM and $D$-module Laplace transform}
For any $q\in\bbN$, the topological Laplace transform of $\cH^{q-1}\Pi (Rf_*\bbC[\dim Y])$ is isomorphic to the  $\bbC$-co-Stokes structure associated to the vector bundle with connection
\begin{equation}\label{eq: GM system u inverse}
    \Big(\mathbb{H}^{q+\dim Y}\big(Y, (\Omega_Y^{\bullet}[u^{-1}], d-d f /u)\big), \nabla_{\partial_u}\coloneqq\partial_u+f/u^2\Big) .
\end{equation}
\end{lemma}
\begin{proof}
The proof builds upon the idea of \cite[Theorem 6.3]{sabbah_non-commutative_hodge_structures}, particularly the analysis of cohomologically tame functions on affine varieties. 
By \cite[\S 1.1]{Sabbah_on_a_twisted_de_rham_complex_I}, we have \begin{align*}
\big(\mathbb{H}^{q+\dim Y}(Y, (\Omega_Y^{\bullet}[u^{-1}], d-d f /u)), \nabla_{\partial_u}\big) &= i_* j^* \FT(\cH^q Rf_* \cO_Y )  \\ &=i_* j^* j_*j^*\FT(\cH^q f_* \cO_Y )     \\ &= i_*j^* \FT\big(\Pi(\cH^q Rf_*\cO_Y)\big). 
\end{align*}
Then, by the compatibility between the topological Laplace transform $\Psi$ and the $D$-module Fourier transform (see \cite[\S 7]{Sabbah_introduction_to_Stokes_structure}), the $\bbC$-co-Stokes-structure associated to \eqref{eq: GM system u inverse} is given by 
\begin{equation}\label{eq: functorial construction of F}
    \Psi\big(\cH^{-1}\DR(\Pi(\cH^qRf_*\cO_Y))\big),
\end{equation}
where $\DR$ denotes the de Rham functor.
Recall that the projector $\Pi$ is compatible with $\DR$, so we have 
\[\DR\big(\Pi(\cH^q Rf_*\cO_Y)\big)=\Pi\big(\DR(\cH^q Rf_*\cO_Y)\big).\]
Since the de Rham functor is t-exact with respect to the perverse t-structure on the bounded derived category of constructible sheaves,  we have 
\[\DR(\cH^q Rf_*\cO_Y)= {}^p{\cH^q}(\DR(Rf_*\cO_Y)).\]
By \cite[Theorem 7.1.1]{HTT_d_module_perverse_sheaves_representation_theory}, for regular holonomic $D$-modules, the de Rham functor is compatible with pushforward,
\[\DR(Rf_*\cO_Y)=Rf_*(\DR\cO_Y).\]
Note that $\DR(\cO_Y)=\bbC[\dim Y]$, we conclude that \[\DR\big(\Pi(\cH^q(Rf_*\cO_Y))\big) = \Pi \big({}^p{\cH^q}  (Rf_*\bbC[\dim Y])\big),\]
which is a constructible sheaf placed in degree $-1$.
By \cite[Proposition 2.4.3 (2), Corollary 3.2.3]{Fresan_Exponential_Motives}, we have
\begin{equation}\label{eq: this sheaf lies in constrz}
\Pi (\prescript{p}{}{{\cH}}^q  Rf_*\bbC)= \cH^{q-1} \Pi Rf_*\bbC[1].
\end{equation}
So we conclude from \eqref{eq: functorial construction of F} that
\[\Psi\big(\cH^{q-1}\Pi (Rf_*\bbC[\dim Y])\big)\] is isomorphic to the $\bbC$-co-Stokes structure associated to \[\Big(\mathbb{H}^{q+\dim Y}\big(Y, (\Omega_Y^{\bullet}[u^{-1}], d-d f /u)\big), \nabla_{\partial_u}\Big).\]
\end{proof}

Next, we prove that the co-Stokes structure in \cref{lemma: GM and $D$-module Laplace transform} can also be described via rapid decay cohomology.

\begin{lemma}\label{rapid decay cohomology and graph}
Let $A$ be a contractible subspace of $\bbA^1$ and $\Gamma\subset Y\times \bbA^1$ the graph of $f \colon Y\rightarrow \bbA^1$.
The projection map
\[p\colon\big(Y\times A, \Gamma\cap (Y\times A)\big)\longrightarrow \big(Y,f^{-1}(A)\big)\]
induces an isomorphism of relative cohomology groups
\[
H^q\big(Y\times A, \Gamma\cap (Y\times A);\bbQ\big)\xrightarrow{\ \sim \ } H^q\big(Y,f^{-1}(A);\bbQ\big).
\]\end{lemma}

\begin{proof}
The map $p$ induces the following map of long exact sequences of relative cohomology.
\begin{equation*}
    \begin{tikzcd}[column sep=small]
{} \arrow[r] &   H^q\big(Y\times A, \Gamma\cap (Y\times A);\bbQ\big) \arrow[r] \arrow[d] &   H^q\big(Y\times A;\bbQ\big)\arrow[r] \arrow[d] & H^q\big(\Gamma\cap (Y\times A);\bbQ\big)  \arrow[r] \arrow[d] & {} \\
{} \arrow[r] & {H^q\big(Y,f^{-1}(A);\bbQ\big)} \arrow[r]                   & {H^q(Y;\bbQ)} \arrow[r]       & {H^q\big(f^{-1}(A);\bbQ\big)} \arrow[r]                                   & {}
\end{tikzcd}
\end{equation*}
Since $A$ is contractible, the middle vertical arrow is an isomorphism. 
Since $p\colon\Gamma\cap (Y\times A) \longrightarrow f^{-1}(A)$ is a homeomorphism with an inverse given by $s(x)=(x,f(x))$, the right vertical arrow is an isomorphism. 
Thus, the left vertical arrow is also an isomorphism by the five lemma.
\end{proof}
For any closed subset $Z\subset X$, let $\alpha\colon X\setminus Z\hookrightarrow X$ be the inclusion of the complement, and denote $k_{[X,Z]}\coloneqq \alpha_!\alpha^*k_X$.
\begin{lemma}\label{lemma: describe stalk}
For any $q\in\bbN$, the stalk of $\cH^{q-1}\Pi (Rf_*k_Y)$ at $\xi$ is isomorphic to $H^q(Y,f^{-1}(\xi);k)$.
\end{lemma}
\begin{proof}
Let $\beta\colon Y\times \bbA^1\setminus \Gamma \rightarrow  Y\times \bbA^1$ denote the open embedding and $\gamma\colon Y\times \bbA^1\rightarrow \bbA^1$ the projection. By \cite[Proposition 3.2.2]{Fresan_Exponential_Motives}, we  have 
\begin{equation}\label{eq: (1) change pushforward to bracket}
    \cH^{q-1}\Pi (Rf_*k_Y) = R^q\gamma_*k_{[Y\times \bbA^1, \Gamma]}.
\end{equation}
Consider the following commutative diagram.
\[ \begin{tikzcd}
{Y\times  \{\xi\}} \arrow[r, "i'"] \arrow[d, "\gamma'"'] & Y\times \bbA^1 \arrow[d, "\gamma"] \\
{\{\xi\}} \arrow[r, "i"]                                     & \bbA^1                             
\end{tikzcd}
\]
By the proper base change theorem, we have
\begin{equation}\label{eq: (2) computation of f over closed planes}
i^*R\gamma_*k_{[Y\times \bbA^1, \Gamma]} = R\gamma'_*i'^*k_{[Y\times \bbA^1, \Gamma]} = R\gamma'_*i'^*\beta_!\beta^*k_{Y\times \bbA^1}.
\end{equation}
Consider the following diagram.
\begin{equation*}
    \begin{tikzcd}
{Y\times  \{\xi\} \setminus \Gamma \cap \big(Y\times  \{\xi\}}\big) \arrow[r, "\beta'"] \arrow[d, "k"'] & {Y\times \{\xi\} } \arrow[d, "i'"] \\
Y\times\bbA^1 \setminus \Gamma \arrow[r, "\beta"]                                                                             & Y\times \bbA^1                                      
\end{tikzcd}
\end{equation*}
By \cite[Theorem 2.3.26]{Dimca_Sheaves_in_Topology}, the right hand side of \eqref{eq: (2) computation of f over closed planes} is
\begin{equation}\label{eq: computatoin (3)}
   R\gamma'_*i'^*\beta_!\beta^*k_{Y\times \bbA^1} = R\gamma'_* \beta'_!k^*\beta^*k_{Y\times \bbA^1}  = R\gamma'_*k_{[Y\times  \{\xi\} , \Gamma\cap Y\times  \{\xi\} ]}.
\end{equation}
Then \eqref{eq: (1) change pushforward to bracket} implies that 
\[\big(\cH^{q-1}\Pi (Rf_*k_Y)\big)_\xi=H^q(Y,f^{-1}(\xi);k).\]
\end{proof}

Next we describe the sections of $\cH^{q-1}\Pi (Rf_*k_Y)$ over a half-plane $\cH_{r,\theta}$ for some $r>0$ and $\theta\in S^1$.
Fix a nonsingular point $c_0 \in \cH_{r,\theta}$ with argument $\theta$.
For each singular point $c_i$, we choose a small open disk $D_i$ around $c_i$, and a path $a_i$ from $c_0$ to a boundary point on $D_i$.
Let $\Gamma_i$ be a small open neighborhood of $\oD_i\cup a_i$. See \cite[Figure 4]{KKP_Hodge_theoretic_aspect_of_mirror_symmetry}.

\begin{lemma}\label{lemma: vanishing cycle at ci}
The complex of vanishing cycles of $\cH^{q-1}\Pi (Rf_*k_Y)$ at $c_i$ satisfies 
\[ 
\phi_{c_i}(\cH^{q-1}\Pi (Rf_*k_Y))\simeq H^q(f^{-1}(\Gamma_i),f^{-1}(c_0);k).
\]
\end{lemma}
\begin{proof}
Consider the exact sequence of homology groups of the triple 
$f^{-1}(c_0)\subset f^{-1}(\Gamma_i)\subset Y$,
\[\xrightarrow{ \ \ \ } H_q(f^{-1}(\Gamma_i), f^{-1}(c_0);k)\xrightarrow{\ \iota \ } H_q(Y,f^{-1}(c_0);k) \xrightarrow{\ \ \ } H_q(Y,f^{-1}(\Gamma_i);k) \xrightarrow{ \ \ \ } .\]
Since $\bigcup f^{-1}(\Gamma_i)$ is a deformation retract of $Y$ (see \cite[\S 5.4]{Liviu_an_invitation_to_morse_theory}), the map $\iota$ is injective.
It induces the following exact sequence of cohomology groups,
\begin{equation}\label{eq: long exact sequence of cohomology}
    \xrightarrow{ \ \ \ } H^q(Y,f^{-1}(\Gamma_i);k)\xrightarrow{ \ \ \ } H^q(Y,f^{-1}(c_0);k)\xrightarrow{\ \kappa \ } H^q(f^{-1}(\Gamma_i), f^{-1}(c_0);k)\xrightarrow{\ \ \ },
\end{equation}
where $\kappa$ is surjective.
A path $a_i$ identifies the stalk of $\cH^{q-1}\Pi (Rf_*k_Y)$ at $c_i$ as a subspace of that at $c_0$ via the composition
\begin{equation}\label{eq: Fc0 quotient Fci}
    H^q(Y,f^{-1}(c_i);k)\xleftarrow{\ \sim \ }H^q(Y,f^{-1}(\Gamma_i);k)\xrightarrow{\ \ \ } H^q(Y,f^{-1}(c_0);k).
\end{equation}
Since $\kappa$ is surjective,  \eqref{eq: long exact sequence of cohomology} and \eqref{eq: Fc0 quotient Fci} imply that
\[\phi_{c_i}(\cH^{q-1}\Pi (Rf_*k_Y))\simeq H^q(f^{-1}(\Gamma_i),f^{-1}(c_0);k).
\]
\end{proof}

\begin{lemma}\label{lemma: comparison of Betti data} 
For any $q\in\bbN$, the Stokes structure associated to the topological Laplace transform of $\cH^{q-1}\Pi (Rf_*k_Y)$ is 
\begin{align*}
    \cL_{k,\theta} & = H^q(Y,f^{-1}\big(\cH_{\infty, \theta});k\big),\\
    \cL_{k,\leq r, \theta} & = H^q\big(Y,f^{-1}(\cH_{r,\theta});k\big).
\end{align*}
\end{lemma}

\begin{proof}
Denote $\cF \coloneqq \cH^{q-1}\Pi (Rf_*k_Y)\in \Constrz$. The $k$-Stokes structure $(\cL_{k}, \cL_{k, \leq})$ associated to the topological Laplace transform $\Psi(\cF)$ is given by
\begin{align*}
    \cL_{k,\theta} &= \cF(\cH_{\infty, \theta}), \\
    \cL_{k,\leq r, \theta} &= \cF(\cH_{r,\theta}).
\end{align*}
Consider the vanishing cycle decomposition of $\cF$ at $c_0$,
\[\cF_{c_0}\xrightarrow{\ \sim \ } \bigoplus_{i=1}^n \cF_{c_0}/\cF_{c_i}.\]
Then by \cref{lemma: vanishing cycle at ci}, the image of $\cF({\cH_{r,\theta}})$ in $\cF_{c_0}$ is 
\begin{equation}\label{eq: F over halfplane}
\bigoplus_{c_i\notin\cH_{r,\theta}}\cF_{c_0}/\cF_{c_i}= H^q(f^{-1}(\cup_{c_i\notin \cH_{r,\theta}} \Gamma_i),f^{-1}(c_0);k).
 \end{equation}
Since $\bigcup_{i=1}^n f^{-1}(\Gamma_i)$ is a deformation retract of $Y$ and $\bigcup_{c_i\in\cH_{r,\theta}} f^{-1}(\Gamma_i)$ is a deformation retract of $f^{-1}(\cH_{r,\theta})$, the long exact sequence of relative cohomology of the triple $f^{-1}(c_0)\subset f^{-1}(\cH_{r,\theta})\subset Y$ implies that the right hand side of \eqref{eq: F over halfplane} satisfies 
\[H^q\big(f^{-1}(\cup_{c_i\notin \cH_{r,\theta}} \Gamma_i),f^{-1}(c_0);k\big) = H^q\big(Y,f^{-1}(\cH_{r,\theta});k\big),\]
completing the proof.
\end{proof}

\begin{proposition}\label{proposition: de Rham and Betti data compatible}
The B-model nc-Hodge structure in \cref{definition: B model nc-Hodge} satisfies the $\bbQ$-structure axiom, i.e.\ let $(\cL_\bbC,\cL_{\bbC,<})$ denote the $\bbC$-co-Stokes structure associated to
\[\big(\mathbb{H}^q(Y, (\Omega_Y^{\bullet}[u^{-1}], d-d f /u)), \nabla_{\partial_u}\big),\] then we have \[(\cL_{\bbC,<}\cap \cL)\otimes_\bbQ \bbC=\cL_{\bbC,<}.
\]
\end{proposition}

\begin{proof}
This follows from \cref{lemma: GM and $D$-module Laplace transform} and \cref{lemma: comparison of Betti data}.
\end{proof}

Let $\Crit(f)$ denote the set of critical points of $f\colon Y\to\bbA^1$, and $\{c_1,\dots,c_n\}\subset\bbC$ the set of critical values.
For $i=1,\dots,n$, let $Z_i \coloneqq f^{-1}(c_i)\cap \mathrm{Crit}(f)$ and let $U^\formal(Z_i)$ be a formal neighborhood of $Z_i$.

\begin{proposition} \label{gluing is dual to GM}
The dual description of the B-model nc-Hodge structure in \cref{definition: B model nc-Hodge} under the equivalence in \cref{theorem: dual description of nc-Hodge is an equivalence} consists of data $\big(\{H_i,\iso_i\},\cF\big)$ defined as follows:
\begin{enumerate}[wide]
    \item $H_i\coloneqq \mathbb{H}^\bullet
    \big(U^\formal(Z_i), (\Omega_X^\bullet\dbb{u},ud-df)\big).$
    \item $\cF\coloneqq \cH^\bullet\Pi (Rf_*\bbQ[\dim Y])$. 
    \item $\iso_i$ is given by the local comparison isomorphism as in \cite[Lemma 3.11]{KKP_Hodge_theoretic_aspect_of_mirror_symmetry}.
\end{enumerate}
\end{proposition}

\begin{proof}
The comparison on Betti data follows from \cref{lemma: comparison of Betti data}.
The comparison between the lattices $H$ and $H_i$ follows from the de Rham global to local isomorphism \cite[Proposition 2.3.4(d)]{Kontsevich_Soibelman_holomorphic_floer_theory} and \cite[Lemma 2.37]{KKP_Hodge_theoretic_aspect_of_mirror_symmetry}. 
Finally, by \cite[Theorem 1]{Sabbah_kontsevich's_conjecture_on_an_algebraic_formula_for_vanishing_cycles_of_local_systems}, the isomorphism $\iso$ in \cref{definition: B model nc-Hodge} splits into the isomorphisms $\iso_i$, completing the proof.
\end{proof}

\begin{corollary} \label{cor:B-model nc-Hodge gluging description}
    If we decompose the B-model nc-Hodge structure in \cref{definition: B model nc-Hodge} by the gluing theorem in \cite[Theorem 2.35]{KKP_Hodge_theoretic_aspect_of_mirror_symmetry}, we obtain the construction in Section 3.2 of loc.\ cit.
\end{corollary}
\begin{proof}
The construction in \cite[\S 3.2]{KKP_Hodge_theoretic_aspect_of_mirror_symmetry} agrees with the dual description in \cref{gluing is dual to GM}, except for the Betti data.
It follows from \cref{lemma: vanishing cycle at ci} that the linear algebraic data of $\cF$ via a choice of Gabrielov paths agree with those of \cite[\S 3.2]{KKP_Hodge_theoretic_aspect_of_mirror_symmetry}. 
\end{proof}

\section{Stokes decomposition of a constructible sheaf}

In this section, we apply the topological Laplace transform functor to the spectral decomposition of nc-Hodge structures, and we relate the spectral decomposition to the vanishing cycle decomposition given by choices of Gabrielov paths.

As an intermediate step, we introduce the decomposition given by the Stokes structure over a good open interval (see \cref{definition: Stokes direction and small}) in \cref{sec:Stokes decomposition}.
We compare the vanishing cycle decomposition with the Stokes decomposition at a point near infinity in \cref{proposition: Stokes decomposition agrees with KKP}, and with the spectral decomposition in \cref{theorem: spectral decomposition and vanishing cycle decomposition}.

Let $(\cL, \cL_<)$ be a co-Stokes structure of exponential type with exponents $\{c_1,\dots, c_n\}$.
We denote by $\cF =\Phi (\cL, \cL_<)$ its topological Laplace transform. 
For $\theta \in S^1$, denote
\[I_\theta\coloneqq(\theta-\pi/2, \theta+\pi/2)\subset S^1.\] 

\subsection{The Stokes decomposition} \label{sec:Stokes decomposition}

\begin{lemma}\label{lemma: computation of H1(I jC) and H1(S1 jC)}
Let $a<b<c<d$ be real numbers, $i\colon (a,b)\hookrightarrow I\coloneqq (a,c)$ and $j\colon (b,c)\hookrightarrow J\coloneqq (a,d)$ open embeddings.
Denote by $k$ the open embedding of an arc in $S^1$.
Then we have the following.
\begin{enumerate}[wide]
    \item $H^*(I,i_!\bbC)=0$.
    \item $H^1(J,j_!\bbC)=\bbC$ and $H^q(J,j_!\bbC)=0$ for $q\neq 1$.
    \item $H^1(S^1, k_!\bbC) = \bbC$ and $H^q(S^1,k_!\bbC)=0$ for $q\neq 1$.
\end{enumerate}
\end{lemma}
\begin{proof}
    Each computation follows from the fiber sequence associated to the open embedding and its complement.
\end{proof}

\begin{proposition}\label{proposition: Stokes decomposition}
Let $\theta\in S^1$ and assume that $I_\theta$ is a good open interval (see \cref{definition: Stokes direction and small}). Let $\beta \colon I_\theta\hookrightarrow S^1$ denote the inclusion.
For any $\xi\in\bbC$, let $\beta_{c_i<\xi} \colon\{x\in I_\theta| c_i<_x\xi\}\hookrightarrow I_\theta$ denote the open inclusion (see \cref{notation: xi1 smaller than xi2}), and $W_i\coloneqq H^1(S^1, \beta_!\beta_{c_i< \xi,!}\beta_{c_i< \xi}^*\beta^*\gr_{c_i}\cL)$.
Note $W_i=0$ if $\xi = c_i$. The unique trivialization of $(\cL, \cL_<)$ over $I_\theta$ induces a decomposition
\[ \cF_\xi = \bigoplus_{i=1}^n W_i,\]
which we call the \emph{Stokes decomposition} of $\cF_\xi$ in the $\theta$-direction.
\end{proposition}

\begin{proof}
If $\xi = c_i$, then $\{x\in I_\theta| c_i<_x\xi\} = \emptyset$, hence $W_i=0$.
Let $\alpha \colon J_\theta\hookrightarrow S^1$ be the closed complement of $I_\theta$. By \cref{proposition: Stokes filtration trivializes over good open intervals}, there exists a unique isomorphism  \[ \bigoplus_{i=1}^n \beta_{c_i< \xi,!}\beta_{c_i< \xi}^*\beta^*\gr_{c_i}\cL\xrightarrow{\ \sim \ }\beta^*\cL_{<\xi}.\] 
This induces an isomorphism of cohomology groups
\begin{equation}\label{eq: decompose H1 usings stokes}
\bigoplus_{i=1}^n H^1(S^1, \beta_!\beta_{c_i< \xi,!}\beta_{c_i< \xi}^*\beta^*\gr_{c_i}\cL)\xrightarrow{\ \sim \ }H^1(S^1,\beta_!\beta^*\cL_{<\xi}).
\end{equation}
The exact sequence of sheaves \begin{equation*}
0\longrightarrow \beta_!\beta^*\cL_{<\xi}\longrightarrow\cL_{<\xi}\longrightarrow \alpha_*\alpha^*\cL_{<\xi} \longrightarrow 0,
\end{equation*}
induces the exact sequence of cohomology
\begin{equation*}
H^0(S^1, \alpha_*\alpha^*\cL_{<\xi}) \longrightarrow H^1(S^1, \beta_!\beta^*\cL_{<\xi})\longrightarrow H^1(S^1, \cL_{<\xi})\longrightarrow H^1(S^1,\alpha_*\alpha^*\cL_{<\xi}).
\end{equation*}
By \cref{lemma: computation of H1(I jC) and H1(S1 jC)}, we have $H^0(S^1, \alpha_*\alpha^*\cL_{<\xi})= H^1(S^1, \alpha_*\alpha^*\cL_{<\xi})=0$. This results in an isomorphism
\begin{equation}\label{eq: ismorphism of resctrition of Lxi and graded pieces}
    H^1(S^1, \beta_!\beta^*\cL_{<\xi})\xrightarrow{\ \sim \ } H^1(S^1, \cL_{<\xi}) = \cF_{\xi}.
\end{equation}
Combining  \eqref{eq: decompose H1 usings stokes} and \eqref{eq: ismorphism of resctrition of Lxi and graded pieces}, we obtain a decomposition of $\cF_{\xi}$ as follows 
\[     \bigoplus_{i=1}^n H^1(S^1, \beta_!\beta_{c_i< \xi,!}\beta_{c_i< \xi}^*\beta^*\gr_{c_i}\cL)\xrightarrow{\ \sim \ }H^1(S^1,\beta_!\beta^*\cL_{<\xi})\xrightarrow{\ \sim \ }\cF_\xi.
\]
\end{proof}

\subsection{Comparison with the vanishing cycle decomposition}

Fix $\theta\in S^1$ that is not an anti-Stokes direction for $(\cL,\cL_<)$.
Then $\theta+\pi/2$ and $\theta-\pi/2$ are not Stokes directions, and $I_\theta$ is a good open interval.
Let $c_0$ be a point with argument $\theta$ near infinity, and connect each $c_i$ with $c_0$ by a straight line segment.
Then we have a vanishing cycle decomposition as in \cref{def: vanishing cycle decomposition},
\[\cF_{c_0}\xrightarrow{\ \sim \ } \bigoplus_{i=1}^n \cF_{c_0}/\cF_{c_i}.\]

To compare the Stokes decomposition with the vanishing cycle decomposition of $\cF_{c_0}$, we first prove in \cref{proposition: at theta  Fc contains Wi} that the parallel transport along $\theta$ preserves the Stokes decomposition.

Let $S_\xi$ be as in \cref{notation: circles}, and $I_\xi\subset S_\xi$ the image of $I_\theta=(\theta-\pi/2,\theta+\pi/2)\subset S^1$ under the homeomorphism \begin{align*}
    S^1 &\longrightarrow S^1_\xi \\
    \theta &\longmapsto \re(\xi e^{-i\theta}).
\end{align*}
Let $S_{\gamma_i}\coloneqq \bigcup_{\xi\in{\gamma_i}} S_\xi\subset \bbC^*$ and $\beta\colon I=\bigcup_{\xi \in\gamma_i}I_\xi \hookrightarrow S_{\gamma_i}$.

\begin{proposition} \label{proposition: at theta  Fc contains Wi}
Let $d_i \in \bbC$ such that $d_i - c_i$ has argument $\theta$.
Consider the inclusion $\cF_{c_i}\subset\cF_{d_i}$ induced by the parallel transport along the line segment $\gamma_i \coloneqq [c_i, d_i] \subset \bbC$.
It is equal to the subspace $\bigoplus_{k\neq i}W_k\subset \cF_{d_i}$ from the Stokes decomposition of $\cF_{d_i}$ in the $\theta$-direction.
\end{proposition}

\begin{proof}
Let $\cG\coloneqq \sigma(\mathcal{L, \cL_<})$ be the associated co-Stokes sheaf (see \cref{lemma: sigma functor}).
It follows from \cref{compute F over open set} that $\cF(\gamma_i)=H^1(\cup_{\xi\in{\gamma_i}} S_\xi, \cG)$.
Since $\theta$ is not an anti-Stokes direction, there is no $c_j$ other than $c_i$ on the segment $\gamma_i$. 
Thus, the inclusion $\cF_{c_i}\subset\cF_{d_i}$ is induced by the restrictions
\[\cF_{c_i}=H^1(S_{c_i}, \cG)\xleftarrow{\ \sim \ }{}\cF(\gamma_i)=H^1(\cup_{\xi\in{\gamma_i}} S_\xi, \cG)\xrightarrow{\ \ \ } \cF_{d_i}=H^1(S_{d_i}, \cG). \]

Since $\gamma_i$ has an angle $\theta$ with the positive real axis, all the circles $S_\xi$ with $\xi \in \gamma_i$ intersect at two points, with arguments $\theta+\pi/2$ and $\theta-\pi/2$ (see \cref{figure: S^1 along a straight line}).
\begin{figure}[!ht]
   \centering
    \includegraphics[scale=0.6]{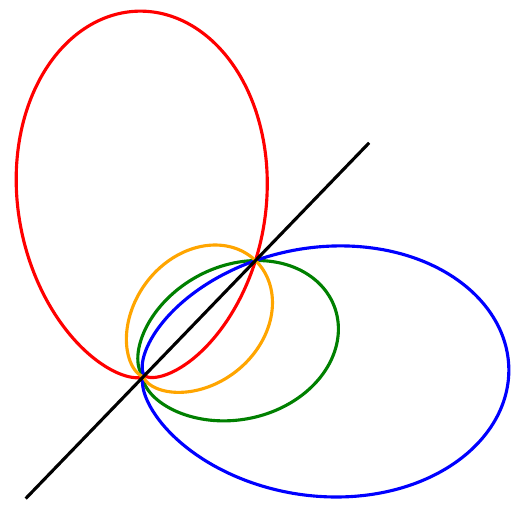}
    \caption{$S_\xi$ along $\xi \in \gamma_i$.}
    \label{figure: S^1 along a straight line}
\end{figure}

Denote by $\iota \colon S_{d_i}\hookrightarrow  S_{\gamma_i}$ the closed embedding.  
By the proper base change theorem, we have \begin{equation}\label{eq: some restriction}
    H^1(S_{\gamma_i}, \cG)= \cF(\gamma_i)\xrightarrow{\ \sim \ }\cF_{c_i} \quad\text{and}\quad H^1(S_{d_i}, \cG)=\cF_{d_i}.
\end{equation}
Consider the following commutative diagram of sheaves on $S_{\gamma_i}$.
\begin{equation*}
    \begin{tikzcd}
\beta_!\beta^*\cG \arrow[r] \arrow[d] & \cG \arrow[d] \\
\iota_*\iota^*\beta_!\beta^*\cG \arrow[r]     & \iota_*\iota^*\cG   
\end{tikzcd}
\end{equation*}
By \eqref{eq: some restriction}, this diagram induces a commutative diagram as follows.
\begin{equation*}
   \begin{tikzcd}
H^1(S_{\gamma_i},\beta_!\beta^*\cG )\arrow[r, "s"] \arrow[d, "t"'] & \cF (\gamma_i) \arrow[d] \\
H^1(S_{\gamma_i},\iota_*\iota^*\beta_!\beta^*\cG) \arrow[r,"r"']     & \cF_{d_i}          
\end{tikzcd}
\end{equation*}    

We now show that the top horizontal arrow $s$ is an isomorphism and the left vertical arrow $t$ is injective. 
Let $\alpha\colon \cup_{\xi\in{\gamma_i}} J_\xi \hookrightarrow S_{\gamma_i}$ denote the closed complement of $\beta$ and $j\colon  S_{\gamma_i} \setminus S_{c_i}\longrightarrow S_{\gamma_i}$ the open complement of $\iota$. 
By \cref{proposition: Stokes filtration trivializes over good open intervals} and \cref{lemma: calculation of cohomology}, $H^*(S_{\gamma_i}, \alpha_*\alpha^*\cG)=0$. We conclude by the exact sequence of cohomology that $s$ is an isomorphism. 
\cref{proposition: Stokes filtration trivializes over good open intervals} implies that $\beta_!\beta^*j_!j^*\cG$ is a direct sum of extensions by zero of constant sheaves in regions indicated in \cref{figure: region of G}, and zero otherwise. 
Then we compute that $H^1(S_{\gamma_i}, \beta_!\beta^*j_!j^*\cG)=0$ and conclude that $t$ is injective. 
By \cref{proposition: Stokes decomposition}, $r$ is identified with the inclusion
\[\bigoplus_{k\neq i} W_k\longrightarrow \bigoplus_{i=1}^n W_k=\cF_{c_0}.\] By the commutativity of the diagram and the comparison of dimension, we conclude that the image of $\cF_{c_i}$ in $\cF_{d_i}$ is $\bigoplus_{k\neq i} W_k\subset \cF_{d_i}$.

\begin{figure}[!ht]
   \centering
    \includegraphics[width=0.8\textwidth]{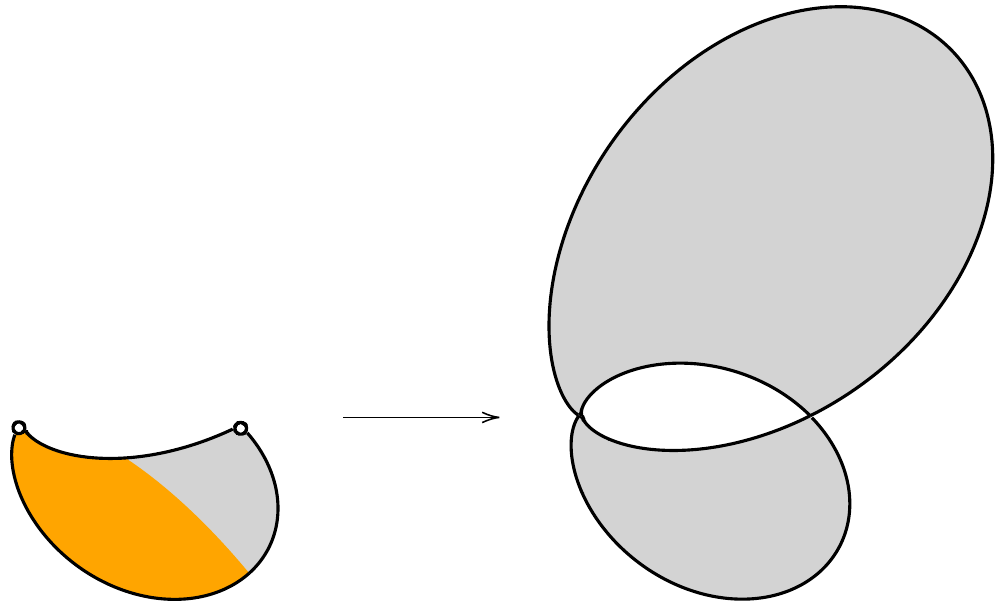}
    \caption{The orange region is the support of a direct summand of $\beta_!\beta^*j_!j^*\cG$.}
    \label{figure: region of G}
\end{figure}
\end{proof}

Let $p$ and $q$ be two non-singular points such that there is no singular point that is colinear with $p$ and $q$.
If $\arg (q-p) =\theta$, similar to the proof of \cref{proposition: at theta  Fc contains Wi}, we can show that the Stokes decompositions of $\cF_p$ and $\cF_q$ are preserved under the parallel transport along $[p,q]$.
Now we study the case where $\theta'\coloneqq\arg (q-p)\neq \theta$.
Let $\phi_i$ (resp.\ $\psi_i)$ denote the angle where $S_{c_i}$ intersect $S_p$ (resp.\ $S_q$). 

\begin{proposition}
Assume that the angle $\theta'+\pi/2$, at which $S_p$ and $S_q$ intersect, lies in $I_\theta$. If for each $i=1,\dots,n$, $\phi_i$ and $\psi_i$ lie in the same connected component of $(\theta,\theta'+\pi/2)\cup (\theta'+\pi/2,\theta+\pi)$, then the parallel transport along $[p,q]$ preserves the Stokes decompositions.
\end{proposition}
\begin{proof}
Let $\gamma_i\coloneqq [p,q]\subset\bbC$. Since there is no singular point between $p$ and $q$, the parallel transport along $\gamma_i$ is given by the restrictions
\begin{equation*}
    \cF_p\xleftarrow{\ \sim \ }\cF(\gamma_i) \xrightarrow{\ \sim \ }\cF_q.
\end{equation*}
Let $\beta_{c_i}:\{z\in I\colon z>S_{c_i}\}\hookrightarrow I$ (see \cref{notation: circles}) denote the open embedding and $\pi:U\rightarrow S^1$ the projection. 
Then we have a unique decomposition
\[\bigoplus_{i=1}^n   
\beta_{c_i,!}\beta_{c_i}^* \beta^*(\pi^*\gr_{c_i}\cL) \xrightarrow{\ \sim \ }\beta^*\cG,\]
which induces the following map
\begin{equation}\label{eq: decomposition of interval}
     \bigoplus_{i=1}^n  H^1\big(S_{\gamma_i},\beta_{!}\beta_{c_i,!}\beta_{c_i}^* \beta^*(\pi^*\gr_{c_i}\cL) \big)\xrightarrow{ \ \sim \ }H^1(S_{\gamma_i},\beta_{!}\beta^*\cG)\xrightarrow{\ l \ }H^1(S_{\gamma_i},\cG)=\cF(\gamma_i).
\end{equation}
Let $\alpha\colon S_{\gamma_i}\setminus I \hookrightarrow S_{\gamma_i}$ denote the closed complement of $\beta$ and $\alpha_{c_i}\colon \{z\in S_{\gamma_i}\setminus I\ |\  z>S_{c_i}\}\hookrightarrow S_{\gamma_i}\setminus I$. 
Then we have an isomorphism 
\begin{equation}\label{eq: decomposition over I}
    \bigoplus_{i=1}^n   
\alpha_{c_i,!}\alpha_{c_i}^* \alpha^*(\pi^*\gr_{c_i}\cL) \xrightarrow{\ \sim \ }\alpha^*\cG.
\end{equation}
Since $\phi_i$ and $\psi_i$ lie in the same connected component of $(\theta,\theta'+\pi/2)\cup (\theta'+\pi/2,\theta+\pi)$, the complement of $\alpha_{c_i}$ is connected. 
Using the isomorphism \eqref{eq: decomposition over I}, we can show that $H^*(S_{\gamma_i},\alpha_*\alpha^*\cG)=0$, which implies that the map $l$ is an isomorphism and it induces a decomposition of $\cF(\gamma_i)$.

Let $\iota: S_p\hookrightarrow S_{\gamma_i}$ denote the closed embedding. We consider the adjunction $\beta_!\beta^*\cG\rightarrow \iota_*\iota^*\beta_!\beta^*\cG$, which induces the following commutative diagram. 
\begin{equation*}
    \begin{tikzcd}
\bigoplus_{i=1}^n{\beta_!\beta_{c_i,!}\beta_{c_i}^*\beta^*(\pi^*\gr_{c_i}\cL)} \arrow[d] \arrow[r] & \beta_!\beta^*\cG \arrow[d]     \\
\bigoplus_{i=1}^n{\iota_*\iota^*\beta_!\beta_{c_i,!}\beta_{c_i}^*\beta^*(\pi^*\gr_{c_i}\cL)} \arrow[r]  & \iota_*\iota^*\beta_!\beta^*\cG
\end{tikzcd}
\end{equation*}
This induces a commutative diagram of $H^1$.
\begin{equation*}
    \begin{tikzcd}
\bigoplus_{i=1}^n H^1\big(S_{\gamma_i},{\beta_!\beta_{c_i,!}\beta_{c_i}^*\beta^*(\pi^*\gr_{c_i}\cL)}\big) \arrow[d, "t"] \arrow[r] & \cF(\gamma_i) \arrow[d]     \\
\bigoplus_{i=1}^nH^1\big(S_{\gamma_i},{\iota_*\iota^*\beta_!\beta_{c_i,!}\beta_{c_i}^*\beta^*(\pi^*\gr_{c_i}\cL)}\big) \arrow[r, "r"]  & \cF_p
\end{tikzcd}
\end{equation*}
By the exact sequence on cohomology induced by the closed embedding $\iota$ and its open complement, we can show that $t$ is an isomorphism preserving each direct summand. Furthermore, it follows from the base change theorem that $r$ gives the Stokes decomposition of $\cF_p$. This implies that the decomposition of $\cF(\gamma_i)$ in \eqref{eq: decomposition of interval} and the Stokes decomposition of $\cF_p$ agree.
The same argument works for $\cF_q$, and we conclude that the Stokes decompositions of $\cF_p$ and $\cF_q$ agree.
\end{proof}

\begin{proposition}\label{proposition: Stokes decomposition invariant near infinity}
Let $p$ and $q$ be two points in $\bbC$ such that for any $i=1,\dots,n$, we have $p>_\theta c_i$ and $q>_\theta c_i$. Then the Stokes decompositions at $\cF_{p}$ and $\cF_{q}$ agree.
\end{proposition}
\begin{proof}
We have shown that the Stokes decompositions at $\cF_{p}$ and $\cF_{q}$ agree in the following two cases:
\begin{enumerate}
    \item The difference $q - p$ has argument $\theta$.
    \item The difference $q - p$ has argument $\theta' \in (\theta-\pi,\theta)$, and the difference between the arguments of $p$ and $q$ is small enough such that for each $1\leq i\leq n$, $\phi_i$ and  $\psi_i$ lie in the same connected component of $(\theta,\theta'+\pi/2)\cup(\theta'+\pi/2,\theta+\pi)$.
\end{enumerate}
Using the first condition, we can push $p$ and $q$ along $\theta$ to $p'$ and $q'$, so that $p'$ and $q'$ will satisfy the second condition, completing the proof.
\end{proof}

\begin{proposition}\label{proposition: Stokes decomposition agrees with KKP}
Let $\cF\in \Constrz$ with singularities $\{c_1,\dots, c_n\}$ and $(\cL, \cL_<)\coloneqq \Psi(\cF)$.
Let $c_0$ be a point near infinity with argument $\theta$, which is not an anti-Stokes direction.
The following two decompositions of $\cF_{c_0}$ agree.
\begin{enumerate}[wide]
    \item The Stokes decomposition, via trivialization of $(\cL, \cL_<)$ over $I_\theta$ (see \cref{proposition: Stokes decomposition}).
    \item The vanishing cycle decomposition, by choosing Gabrielov paths $\gamma_i$ as straight line segments from $c_i$ to $c_0$ (see \cref{def: vanishing cycle decomposition}).    
\end{enumerate}
\end{proposition}
\begin{proof}
For each point $c_i$, choose $d_i$ such $d_i-c_i$ has argument $\theta$.
For $j=1,\dots, n$, we have the Stokes decompositions
\[    \cF_{d_j} = \bigoplus_{k=1}^n W_k^j, \quad\text{and}\quad     \cF_{c_0} = \bigoplus_{k=1}^nW_k^0.
\]

By \cref{proposition: at theta  Fc contains Wi},  we have $\cF_{c_j} = \bigoplus_{k\neq i} W_k^j$ as a subspace of $\cF_{d_j}$. By \cref{proposition: Stokes decomposition invariant near infinity}, we have $\cF_{c_j}=\bigoplus_{k\neq j} W_k^0$ as a subspace of $\cF_{c_0}$. On the other hand, the vanishing cycle decomposition is
\begin{equation*}
    \psi\colon  \cF_{c_0}  \xrightarrow{\ \sim \ } \bigoplus_{i=1}^n \cF_{c_0}/\cF_{c_i}.
\end{equation*}
Therefore, the isomorphism \[\bigoplus_{i=1}^n W_i^0\xrightarrow{\ \sim \ } \cF_{c_0}\xrightarrow{\ \sim \ } 
\bigoplus_{i=1}^n \cF_{c_0}/\cF_{c_i}\]
maps $W_i^0 = \bigcap_{j \neq i} ( \bigoplus_{k\neq j} W_k^0 ) = \bigcap_{j \neq i}\cF_{c_j}$ to $\cF_{c_0}/\cF_{c_i}$, i.e.\ it preserves the two decompositions.
\end{proof}

Let $(H, \nabla)$ be an algebraic vector bundle with a connection on $\bbA^1$ having a regular singularity at $\infty$ and an irregular singularity of exponential type at $0$.

\begin{definition}\label{cor: spectral decomposition}
We have a formal isomorphism
\[(H,\nabla)\otimes_{\bbC[u]}\bbC\dbb{u} \simeq \bigoplus_{i=1}^n (\cM_{c_i},\nabla_{c_i}),\]
where $(\cM_{c_i},\nabla_{c_i})$ is a meromorphic $\bbC\dbb{u}$-connection with a pole of order 2, and the endomorphism $\nabla_{u^2\partial_u}|_{u=0}$ has only $c_i$ as its eigenvalue, (see \cite[Theorem 5.7, Exercise 5.9]{Sabbah_isomonodromic_deformations_and_frobenius_manifolds} or \cite[Theorem 11.1]{Wasow_asymptotic_expansion_for_ordinary_differential_equations}).
We call it the \emph{spectral decomposition} of $(H,\nabla)$.
\end{definition}

Next we relate the spectral decomposition and the vanishing cycle decomposition, via asymptotic lifts.

\begin{theorem}[Spectral decomposition and vanishing cycle decomposition]\label{theorem: spectral decomposition and vanishing cycle decomposition}
Let $(H, \nabla)$ be an algebraic vector bundle with a connection on $\bbA^1$ having a regular singularity at $\infty$ and an irregular singularity of exponential type at $0$. The asymptotic lift of the spectral decomposition along any non-anti-Stokes direction $\theta$ agrees with the vanishing cycle decomposition at any point $\xi$ near infinity with argument $\theta$.
\end{theorem}
\begin{proof}
By \cite[Corollary 3.3]{Sanda_Shamoto_an_analogue_of_},
the spectral decomposition of $(H, \nabla)$ is compatible with  Hukuhara-Turrittin-Levelt decomposition, i.e.\ there exists $\bbC\{u\}[u^{-1}]$-modules $R_i$ with regular connections $\nabla_i$, such that \[(\mathcal{M}_{c_i},\nabla_{c_i})\otimes \bbC \dbp{u}=\big(e^{c_i/u}\otimes_{\bbC\{u\}[u^{-1}]}(\mathcal{R}_i, \nabla_i)\big)\otimes \bbC\dbp{u}.\] By \cite[Lemma 8.3]{Hertling_Sevenheck_nilpotent_orbits_of_a_generalization_of_hodge_structures}, this formal decomposition lifts uniquely to an analytic decomposition on an open sector in $\bbC^*$ from angle $\theta-\pi/2$ to $\theta+\pi/2$.
Let $(\cL,\cL_<)$ be the co-Stokes structure associated to $(H,\nabla)$.
Then the above decomposition induces a unique trivialization of $(\cL,\cL_<)$ over $I_\theta\subset S^1$ as \begin{align}
    \label{eq: spectral decomand vanishing 1}\beta^*\cL &= \bigoplus_{i=1}^n \beta^* \gr_{c_i}\cL, \\ \label{eq: spectral decomand vanishing 2}\beta^*\cL_{<\xi}&=\bigoplus_{i=1}^n \beta_{c_i< \xi,!}\beta_{c_i< \xi}^*\beta^*\gr_{c_i}\cL.
\end{align}
Since $\xi$ is near infinity, we have an isomorphism $(\beta^*\cL_{<\xi})_\theta\xrightarrow{\sim}(\beta^*\cL)_\theta$, which is compatible with the decompositions \eqref{eq: spectral decomand vanishing 1} and \eqref{eq: spectral decomand vanishing 2} above.
Recall that the Stokes decomposition is given by
\begin{equation}\label{eq: spectral decomand vanishing 3}
    \cF_\xi = H^1(S^1,\cL_{<\xi}) =H^1(S^1, \beta_!\beta^*\cL_{<\xi})  = \bigoplus_{i=1}^n H^1(S^1, \beta_!\beta_{c_i< \xi,!}\beta_{c_i< \xi}^*\beta^*\gr_{c_i}\cL).
\end{equation}
By a Čech cohomology computation, we obtain that $H^1(S^1,\cL_{<\xi}) = (\beta^*\cL_{<\xi})_\theta$, which is compatible with decompositions \eqref{eq: spectral decomand vanishing 2} and \eqref{eq: spectral decomand vanishing 3}. By \cref{proposition: Stokes decomposition agrees with KKP}, the Stokes decomposition agrees with the vanishing cycle decomposition via a choice of straight line segments as Gabrielov paths. We conclude the proof.
\end{proof}

\bibliographystyle{plain}
\bibliography{reference}

\end{document}